\UseRawInputEncoding
\documentclass[oneside, a4paper,reqno]{amsart}
\usepackage{pdfsync}
\usepackage{stmaryrd}
\usepackage{mathrsfs}
\usepackage{multicol}
\usepackage{amsmath, amsthm, amscd, amssymb, latexsym, eucal}
\usepackage[all]{xy}
\def\serieslogo@{} \def\@setcopyright{} \makeatother

\usepackage{multienum}

\usepackage{hyperref}
\usepackage{color}
\usepackage{cite}

\hypersetup{colorlinks=true,
     breaklinks=true,
     linkcolor=blue,    
     bookmarks=true,
      pageanchor=true}

\makeatletter
\renewcommand*\env@matrix[1][c]{\hskip -\arraycolsep
  \let\@ifnextchar\new@ifnextchar
  \array{*\c@MaxMatrixCols #1}}
\makeatother

\usepackage{color}

\definecolor{AliceBlue}{rgb}{0.94,0.97,1.00}
\definecolor{AntiqueWhite1}{rgb}{1.00,0.94,0.86}
\definecolor{AntiqueWhite2}{rgb}{0.93,0.87,0.80}
\definecolor{AntiqueWhite3}{rgb}{0.80,0.75,0.69}
\definecolor{AntiqueWhite4}{rgb}{0.55,0.51,0.47}
\definecolor{AntiqueWhite}{rgb}{0.98,0.92,0.84}
\definecolor{BlanchedAlmond}{rgb}{1.00,0.92,0.80}
\definecolor{BlueViolet}{rgb}{0.54,0.17,0.89}
\definecolor{CadetBlue1}{rgb}{0.60,0.96,1.00}
\definecolor{CadetBlue2}{rgb}{0.56,0.90,0.93}
\definecolor{CadetBlue3}{rgb}{0.48,0.77,0.80}
\definecolor{CadetBlue4}{rgb}{0.33,0.53,0.55}
\definecolor{CadetBlue}{rgb}{0.37,0.62,0.63}
\definecolor{CornflowerBlue}{rgb}{0.39,0.58,0.93}
\definecolor{DarkBlue}{rgb}{0.00,0.00,0.55}
\definecolor{DarkCyan}{rgb}{0.00,0.55,0.55}
\definecolor{DarkGoldenrod1}{rgb}{1.00,0.73,0.06}
\definecolor{DarkGoldenrod2}{rgb}{0.93,0.68,0.05}
\definecolor{DarkGoldenrod3}{rgb}{0.80,0.58,0.05}
\definecolor{DarkGoldenrod4}{rgb}{0.55,0.40,0.03}
\definecolor{DarkGoldenrod}{rgb}{0.72,0.53,0.04}
\definecolor{DarkGray}{rgb}{0.66,0.66,0.66}
\definecolor{DarkGreen}{rgb}{0.00,0.39,0.00}
\definecolor{DarkGrey}{rgb}{0.66,0.66,0.66}
\definecolor{DarkKhaki}{rgb}{0.74,0.72,0.42}
\definecolor{DarkMagenta}{rgb}{0.55,0.00,0.55}
\definecolor{DarkOliveGreen1}{rgb}{0.79,1.00,0.44}
\definecolor{DarkOliveGreen2}{rgb}{0.74,0.93,0.41}
\definecolor{DarkOliveGreen3}{rgb}{0.64,0.80,0.35}
\definecolor{DarkOliveGreen4}{rgb}{0.43,0.55,0.24}
\definecolor{DarkOliveGreen}{rgb}{0.33,0.42,0.18}
\definecolor{DarkOrange1}{rgb}{1.00,0.50,0.00}
\definecolor{DarkOrange2}{rgb}{0.93,0.46,0.00}
\definecolor{DarkOrange3}{rgb}{0.80,0.40,0.00}
\definecolor{DarkOrange4}{rgb}{0.55,0.27,0.00}
\definecolor{DarkOrange}{rgb}{1.00,0.55,0.00}
\definecolor{DarkOrchid1}{rgb}{0.75,0.24,1.00}
\definecolor{DarkOrchid2}{rgb}{0.70,0.23,0.93}
\definecolor{DarkOrchid3}{rgb}{0.60,0.20,0.80}
\definecolor{DarkOrchid4}{rgb}{0.41,0.13,0.55}
\definecolor{DarkOrchid}{rgb}{0.60,0.20,0.80}
\definecolor{DarkRed}{rgb}{0.55,0.00,0.00}
\definecolor{DarkSalmon}{rgb}{0.91,0.59,0.48}
\definecolor{DarkSeaGreen1}{rgb}{0.76,1.00,0.76}
\definecolor{DarkSeaGreen2}{rgb}{0.71,0.93,0.71}
\definecolor{DarkSeaGreen3}{rgb}{0.61,0.80,0.61}
\definecolor{DarkSeaGreen4}{rgb}{0.41,0.55,0.41}
\definecolor{DarkSeaGreen}{rgb}{0.56,0.74,0.56}
\definecolor{DarkSlateBlue}{rgb}{0.28,0.24,0.55}
\definecolor{DarkSlateGray1}{rgb}{0.59,1.00,1.00}
\definecolor{DarkSlateGray2}{rgb}{0.55,0.93,0.93}
\definecolor{DarkSlateGray3}{rgb}{0.47,0.80,0.80}
\definecolor{DarkSlateGray4}{rgb}{0.32,0.55,0.55}
\definecolor{DarkSlateGray}{rgb}{0.18,0.31,0.31}
\definecolor{DarkSlateGrey}{rgb}{0.18,0.31,0.31}
\definecolor{DarkTurquoise}{rgb}{0.00,0.81,0.82}
\definecolor{DarkViolet}{rgb}{0.58,0.00,0.83}
\definecolor{DeepPink1}{rgb}{1.00,0.08,0.58}
\definecolor{DeepPink2}{rgb}{0.93,0.07,0.54}
\definecolor{DeepPink3}{rgb}{0.80,0.06,0.46}
\definecolor{DeepPink4}{rgb}{0.55,0.04,0.31}
\definecolor{DeepPink}{rgb}{1.00,0.08,0.58}
\definecolor{DeepSkyBlue1}{rgb}{0.00,0.75,1.00}
\definecolor{DeepSkyBlue2}{rgb}{0.00,0.70,0.93}
\definecolor{DeepSkyBlue3}{rgb}{0.00,0.60,0.80}
\definecolor{DeepSkyBlue4}{rgb}{0.00,0.41,0.55}
\definecolor{DeepSkyBlue}{rgb}{0.00,0.75,1.00}
\definecolor{DimGray}{rgb}{0.41,0.41,0.41}
\definecolor{DimGrey}{rgb}{0.41,0.41,0.41}
\definecolor{DodgerBlue1}{rgb}{0.12,0.56,1.00}
\definecolor{DodgerBlue2}{rgb}{0.11,0.53,0.93}
\definecolor{DodgerBlue3}{rgb}{0.09,0.45,0.80}
\definecolor{DodgerBlue4}{rgb}{0.06,0.31,0.55}
\definecolor{DodgerBlue}{rgb}{0.12,0.56,1.00}
\definecolor{FloralWhite}{rgb}{1.00,0.98,0.94}
\definecolor{ForestGreen}{rgb}{0.13,0.55,0.13}
\definecolor{GhostWhite}{rgb}{0.97,0.97,1.00}
\definecolor{GreenYellow}{rgb}{0.68,1.00,0.18}
\definecolor{HotPink1}{rgb}{1.00,0.43,0.71}
\definecolor{HotPink2}{rgb}{0.93,0.42,0.65}
\definecolor{HotPink3}{rgb}{0.80,0.38,0.56}
\definecolor{HotPink4}{rgb}{0.55,0.23,0.38}
\definecolor{HotPink}{rgb}{1.00,0.41,0.71}
\definecolor{IndianRed1}{rgb}{1.00,0.42,0.42}
\definecolor{IndianRed2}{rgb}{0.93,0.39,0.39}
\definecolor{IndianRed3}{rgb}{0.80,0.33,0.33}
\definecolor{IndianRed4}{rgb}{0.55,0.23,0.23}
\definecolor{IndianRed}{rgb}{0.80,0.36,0.36}
\definecolor{LavenderBlush1}{rgb}{1.00,0.94,0.96}
\definecolor{LavenderBlush2}{rgb}{0.93,0.88,0.90}
\definecolor{LavenderBlush3}{rgb}{0.80,0.76,0.77}
\definecolor{LavenderBlush4}{rgb}{0.55,0.51,0.53}
\definecolor{LavenderBlush}{rgb}{1.00,0.94,0.96}
\definecolor{LawnGreen}{rgb}{0.49,0.99,0.00}
\definecolor{LemonChiffon1}{rgb}{1.00,0.98,0.80}
\definecolor{LemonChiffon2}{rgb}{0.93,0.91,0.75}
\definecolor{LemonChiffon3}{rgb}{0.80,0.79,0.65}
\definecolor{LemonChiffon4}{rgb}{0.55,0.54,0.44}
\definecolor{LemonChiffon}{rgb}{1.00,0.98,0.80}
\definecolor{LightBlue1}{rgb}{0.75,0.94,1.00}
\definecolor{LightBlue2}{rgb}{0.70,0.87,0.93}
\definecolor{LightBlue3}{rgb}{0.60,0.75,0.80}
\definecolor{LightBlue4}{rgb}{0.41,0.51,0.55}
\definecolor{LightBlue}{rgb}{0.68,0.85,0.90}
\definecolor{LightCoral}{rgb}{0.94,0.50,0.50}
\definecolor{LightCyan1}{rgb}{0.88,1.00,1.00}
\definecolor{LightCyan2}{rgb}{0.82,0.93,0.93}
\definecolor{LightCyan3}{rgb}{0.71,0.80,0.80}
\definecolor{LightCyan4}{rgb}{0.48,0.55,0.55}
\definecolor{LightCyan}{rgb}{0.88,1.00,1.00}
\definecolor{LightGoldenrod1}{rgb}{1.00,0.93,0.55}
\definecolor{LightGoldenrod2}{rgb}{0.93,0.86,0.51}
\definecolor{LightGoldenrod3}{rgb}{0.80,0.75,0.44}
\definecolor{LightGoldenrod4}{rgb}{0.55,0.51,0.30}
\definecolor{LightGoldenrodYellow}{rgb}{0.98,0.98,0.82}
\definecolor{LightGoldenrod}{rgb}{0.93,0.87,0.51}
\definecolor{LightGray}{rgb}{0.83,0.83,0.83}
\definecolor{LightGreen}{rgb}{0.56,0.93,0.56}
\definecolor{LightGrey}{rgb}{0.83,0.83,0.83}
\definecolor{LightPink1}{rgb}{1.00,0.68,0.73}
\definecolor{LightPink2}{rgb}{0.93,0.64,0.68}
\definecolor{LightPink3}{rgb}{0.80,0.55,0.58}
\definecolor{LightPink4}{rgb}{0.55,0.37,0.40}
\definecolor{LightPink}{rgb}{1.00,0.71,0.76}
\definecolor{LightSalmon1}{rgb}{1.00,0.63,0.48}
\definecolor{LightSalmon2}{rgb}{0.93,0.58,0.45}
\definecolor{LightSalmon3}{rgb}{0.80,0.51,0.38}
\definecolor{LightSalmon4}{rgb}{0.55,0.34,0.26}
\definecolor{LightSalmon}{rgb}{1.00,0.63,0.48}
\definecolor{LightSeaGreen}{rgb}{0.13,0.70,0.67}
\definecolor{LightSkyBlue1}{rgb}{0.69,0.89,1.00}
\definecolor{LightSkyBlue2}{rgb}{0.64,0.83,0.93}
\definecolor{LightSkyBlue3}{rgb}{0.55,0.71,0.80}
\definecolor{LightSkyBlue4}{rgb}{0.38,0.48,0.55}
\definecolor{LightSkyBlue}{rgb}{0.53,0.81,0.98}
\definecolor{LightSlateBlue}{rgb}{0.52,0.44,1.00}
\definecolor{LightSlateGray}{rgb}{0.47,0.53,0.60}
\definecolor{LightSlateGrey}{rgb}{0.47,0.53,0.60}
\definecolor{LightSteelBlue1}{rgb}{0.79,0.88,1.00}
\definecolor{LightSteelBlue2}{rgb}{0.74,0.82,0.93}
\definecolor{LightSteelBlue3}{rgb}{0.64,0.71,0.80}
\definecolor{LightSteelBlue4}{rgb}{0.43,0.48,0.55}
\definecolor{LightSteelBlue}{rgb}{0.69,0.77,0.87}
\definecolor{LightYellow1}{rgb}{1.00,1.00,0.88}
\definecolor{LightYellow2}{rgb}{0.93,0.93,0.82}
\definecolor{LightYellow3}{rgb}{0.80,0.80,0.71}
\definecolor{LightYellow4}{rgb}{0.55,0.55,0.48}
\definecolor{LightYellow}{rgb}{1.00,1.00,0.88}
\definecolor{LimeGreen}{rgb}{0.20,0.80,0.20}
\definecolor{MediumAquamarine}{rgb}{0.40,0.80,0.67}
\definecolor{MediumBlue}{rgb}{0.00,0.00,0.80}
\definecolor{MediumOrchid1}{rgb}{0.88,0.40,1.00}
\definecolor{MediumOrchid2}{rgb}{0.82,0.37,0.93}
\definecolor{MediumOrchid3}{rgb}{0.71,0.32,0.80}
\definecolor{MediumOrchid4}{rgb}{0.48,0.22,0.55}
\definecolor{MediumOrchid}{rgb}{0.73,0.33,0.83}
\definecolor{MediumPurple1}{rgb}{0.67,0.51,1.00}
\definecolor{MediumPurple2}{rgb}{0.62,0.47,0.93}
\definecolor{MediumPurple3}{rgb}{0.54,0.41,0.80}
\definecolor{MediumPurple4}{rgb}{0.36,0.28,0.55}
\definecolor{MediumPurple}{rgb}{0.58,0.44,0.86}
\definecolor{MediumSeaGreen}{rgb}{0.24,0.70,0.44}
\definecolor{MediumSlateBlue}{rgb}{0.48,0.41,0.93}
\definecolor{MediumSpringGreen}{rgb}{0.00,0.98,0.60}
\definecolor{MediumTurquoise}{rgb}{0.28,0.82,0.80}
\definecolor{MediumVioletRed}{rgb}{0.78,0.08,0.52}
\definecolor{MidnightBlue}{rgb}{0.10,0.10,0.44}
\definecolor{MintCream}{rgb}{0.96,1.00,0.98}
\definecolor{MistyRose1}{rgb}{1.00,0.89,0.88}
\definecolor{MistyRose2}{rgb}{0.93,0.84,0.82}
\definecolor{MistyRose3}{rgb}{0.80,0.72,0.71}
\definecolor{MistyRose4}{rgb}{0.55,0.49,0.48}
\definecolor{MistyRose}{rgb}{1.00,0.89,0.88}
\definecolor{NavajoWhite1}{rgb}{1.00,0.87,0.68}
\definecolor{NavajoWhite2}{rgb}{0.93,0.81,0.63}
\definecolor{NavajoWhite3}{rgb}{0.80,0.70,0.55}
\definecolor{NavajoWhite4}{rgb}{0.55,0.47,0.37}
\definecolor{NavajoWhite}{rgb}{1.00,0.87,0.68}
\definecolor{NavyBlue}{rgb}{0.00,0.00,0.50}
\definecolor{OldLace}{rgb}{0.99,0.96,0.90}
\definecolor{OliveDrab1}{rgb}{0.75,1.00,0.24}
\definecolor{OliveDrab2}{rgb}{0.70,0.93,0.23}
\definecolor{OliveDrab3}{rgb}{0.60,0.80,0.20}
\definecolor{OliveDrab4}{rgb}{0.41,0.55,0.13}
\definecolor{OliveDrab}{rgb}{0.42,0.56,0.14}
\definecolor{OrangeRed1}{rgb}{1.00,0.27,0.00}
\definecolor{OrangeRed2}{rgb}{0.93,0.25,0.00}
\definecolor{OrangeRed3}{rgb}{0.80,0.22,0.00}
\definecolor{OrangeRed4}{rgb}{0.55,0.15,0.00}
\definecolor{OrangeRed}{rgb}{1.00,0.27,0.00}
\definecolor{PaleGoldenrod}{rgb}{0.93,0.91,0.67}
\definecolor{PaleGreen1}{rgb}{0.60,1.00,0.60}
\definecolor{PaleGreen2}{rgb}{0.56,0.93,0.56}
\definecolor{PaleGreen3}{rgb}{0.49,0.80,0.49}
\definecolor{PaleGreen4}{rgb}{0.33,0.55,0.33}
\definecolor{PaleGreen}{rgb}{0.60,0.98,0.60}
\definecolor{PaleTurquoise1}{rgb}{0.73,1.00,1.00}
\definecolor{PaleTurquoise2}{rgb}{0.68,0.93,0.93}
\definecolor{PaleTurquoise3}{rgb}{0.59,0.80,0.80}
\definecolor{PaleTurquoise4}{rgb}{0.40,0.55,0.55}
\definecolor{PaleTurquoise}{rgb}{0.69,0.93,0.93}
\definecolor{PaleVioletRed1}{rgb}{1.00,0.51,0.67}
\definecolor{PaleVioletRed2}{rgb}{0.93,0.47,0.62}
\definecolor{PaleVioletRed3}{rgb}{0.80,0.41,0.54}
\definecolor{PaleVioletRed4}{rgb}{0.55,0.28,0.36}
\definecolor{PaleVioletRed}{rgb}{0.86,0.44,0.58}
\definecolor{PapayaWhip}{rgb}{1.00,0.94,0.84}
\definecolor{PeachPuff1}{rgb}{1.00,0.85,0.73}
\definecolor{PeachPuff2}{rgb}{0.93,0.80,0.68}
\definecolor{PeachPuff3}{rgb}{0.80,0.69,0.58}
\definecolor{PeachPuff4}{rgb}{0.55,0.47,0.40}
\definecolor{PeachPuff}{rgb}{1.00,0.85,0.73}
\definecolor{PowderBlue}{rgb}{0.69,0.88,0.90}
\definecolor{RosyBrown1}{rgb}{1.00,0.76,0.76}
\definecolor{RosyBrown2}{rgb}{0.93,0.71,0.71}
\definecolor{RosyBrown3}{rgb}{0.80,0.61,0.61}
\definecolor{RosyBrown4}{rgb}{0.55,0.41,0.41}
\definecolor{RosyBrown}{rgb}{0.74,0.56,0.56}
\definecolor{RoyalBlue1}{rgb}{0.28,0.46,1.00}
\definecolor{RoyalBlue2}{rgb}{0.26,0.43,0.93}
\definecolor{RoyalBlue3}{rgb}{0.23,0.37,0.80}
\definecolor{RoyalBlue4}{rgb}{0.15,0.25,0.55}
\definecolor{RoyalBlue}{rgb}{0.25,0.41,0.88}
\definecolor{SaddleBrown}{rgb}{0.55,0.27,0.07}
\definecolor{SandyBrown}{rgb}{0.96,0.64,0.38}
\definecolor{SeaGreen1}{rgb}{0.33,1.00,0.62}
\definecolor{SeaGreen2}{rgb}{0.31,0.93,0.58}
\definecolor{SeaGreen3}{rgb}{0.26,0.80,0.50}
\definecolor{SeaGreen4}{rgb}{0.18,0.55,0.34}
\definecolor{SeaGreen}{rgb}{0.18,0.55,0.34}
\definecolor{SkyBlue1}{rgb}{0.53,0.81,1.00}
\definecolor{SkyBlue2}{rgb}{0.49,0.75,0.93}
\definecolor{SkyBlue3}{rgb}{0.42,0.65,0.80}
\definecolor{SkyBlue4}{rgb}{0.29,0.44,0.55}
\definecolor{SkyBlue}{rgb}{0.53,0.81,0.92}
\definecolor{SlateBlue1}{rgb}{0.51,0.44,1.00}
\definecolor{SlateBlue2}{rgb}{0.48,0.40,0.93}
\definecolor{SlateBlue3}{rgb}{0.41,0.35,0.80}
\definecolor{SlateBlue4}{rgb}{0.28,0.24,0.55}
\definecolor{SlateBlue}{rgb}{0.42,0.35,0.80}
\definecolor{SlateGray1}{rgb}{0.78,0.89,1.00}
\definecolor{SlateGray2}{rgb}{0.73,0.83,0.93}
\definecolor{SlateGray3}{rgb}{0.62,0.71,0.80}
\definecolor{SlateGray4}{rgb}{0.42,0.48,0.55}
\definecolor{SlateGray}{rgb}{0.44,0.50,0.56}
\definecolor{SlateGrey}{rgb}{0.44,0.50,0.56}
\definecolor{SpringGreen1}{rgb}{0.00,1.00,0.50}
\definecolor{SpringGreen2}{rgb}{0.00,0.93,0.46}
\definecolor{SpringGreen3}{rgb}{0.00,0.80,0.40}
\definecolor{SpringGreen4}{rgb}{0.00,0.55,0.27}
\definecolor{SpringGreen}{rgb}{0.00,1.00,0.50}
\definecolor{SteelBlue1}{rgb}{0.39,0.72,1.00}
\definecolor{SteelBlue2}{rgb}{0.36,0.67,0.93}
\definecolor{SteelBlue3}{rgb}{0.31,0.58,0.80}
\definecolor{SteelBlue4}{rgb}{0.21,0.39,0.55}
\definecolor{SteelBlue}{rgb}{0.27,0.51,0.71}
\definecolor{VioletRed1}{rgb}{1.00,0.24,0.59}
\definecolor{VioletRed2}{rgb}{0.93,0.23,0.55}
\definecolor{VioletRed3}{rgb}{0.80,0.20,0.47}
\definecolor{VioletRed4}{rgb}{0.55,0.13,0.32}
\definecolor{VioletRed}{rgb}{0.82,0.13,0.56}
\definecolor{WhiteSmoke}{rgb}{0.96,0.96,0.96}
\definecolor{YellowGreen}{rgb}{0.60,0.80,0.20}
\definecolor{aliceblue}{rgb}{0.94,0.97,1.00}
\definecolor{antiquewhite}{rgb}{0.98,0.92,0.84}
\definecolor{aquamarine1}{rgb}{0.50,1.00,0.83}
\definecolor{aquamarine2}{rgb}{0.46,0.93,0.78}
\definecolor{aquamarine3}{rgb}{0.40,0.80,0.67}
\definecolor{aquamarine4}{rgb}{0.27,0.55,0.45}
\definecolor{aquamarine}{rgb}{0.50,1.00,0.83}
\definecolor{azure1}{rgb}{0.94,1.00,1.00}
\definecolor{azure2}{rgb}{0.88,0.93,0.93}
\definecolor{azure3}{rgb}{0.76,0.80,0.80}
\definecolor{azure4}{rgb}{0.51,0.55,0.55}
\definecolor{azure}{rgb}{0.94,1.00,1.00}
\definecolor{beige}{rgb}{0.96,0.96,0.86}
\definecolor{bisque1}{rgb}{1.00,0.89,0.77}
\definecolor{bisque2}{rgb}{0.93,0.84,0.72}
\definecolor{bisque3}{rgb}{0.80,0.72,0.62}
\definecolor{bisque4}{rgb}{0.55,0.49,0.42}
\definecolor{bisque}{rgb}{1.00,0.89,0.77}
\definecolor{black}{rgb}{0.00,0.00,0.00}
\definecolor{blanchedalmond}{rgb}{1.00,0.92,0.80}
\definecolor{blue1}{rgb}{0.00,0.00,1.00}
\definecolor{blue2}{rgb}{0.00,0.00,0.93}
\definecolor{blue3}{rgb}{0.00,0.00,0.80}
\definecolor{blue4}{rgb}{0.00,0.00,0.55}
\definecolor{blueviolet}{rgb}{0.54,0.17,0.89}
\definecolor{blue}{rgb}{0.00,0.00,1.00}
\definecolor{brown1}{rgb}{1.00,0.25,0.25}
\definecolor{brown2}{rgb}{0.93,0.23,0.23}
\definecolor{brown3}{rgb}{0.80,0.20,0.20}
\definecolor{brown4}{rgb}{0.55,0.14,0.14}
\definecolor{brown}{rgb}{0.65,0.16,0.16}
\definecolor{burlywood1}{rgb}{1.00,0.83,0.61}
\definecolor{burlywood2}{rgb}{0.93,0.77,0.57}
\definecolor{burlywood3}{rgb}{0.80,0.67,0.49}
\definecolor{burlywood4}{rgb}{0.55,0.45,0.33}
\definecolor{burlywood}{rgb}{0.87,0.72,0.53}
\definecolor{cadetblue}{rgb}{0.37,0.62,0.63}
\definecolor{chartreuse1}{rgb}{0.50,1.00,0.00}
\definecolor{chartreuse2}{rgb}{0.46,0.93,0.00}
\definecolor{chartreuse3}{rgb}{0.40,0.80,0.00}
\definecolor{chartreuse4}{rgb}{0.27,0.55,0.00}
\definecolor{chartreuse}{rgb}{0.50,1.00,0.00}
\definecolor{chocolate1}{rgb}{1.00,0.50,0.14}
\definecolor{chocolate2}{rgb}{0.93,0.46,0.13}
\definecolor{chocolate3}{rgb}{0.80,0.40,0.11}
\definecolor{chocolate4}{rgb}{0.55,0.27,0.07}
\definecolor{chocolate}{rgb}{0.82,0.41,0.12}
\definecolor{coral1}{rgb}{1.00,0.45,0.34}
\definecolor{coral2}{rgb}{0.93,0.42,0.31}
\definecolor{coral3}{rgb}{0.80,0.36,0.27}
\definecolor{coral4}{rgb}{0.55,0.24,0.18}
\definecolor{coral}{rgb}{1.00,0.50,0.31}
\definecolor{cornflowerblue}{rgb}{0.39,0.58,0.93}
\definecolor{cornsilk1}{rgb}{1.00,0.97,0.86}
\definecolor{cornsilk2}{rgb}{0.93,0.91,0.80}
\definecolor{cornsilk3}{rgb}{0.80,0.78,0.69}
\definecolor{cornsilk4}{rgb}{0.55,0.53,0.47}
\definecolor{cornsilk}{rgb}{1.00,0.97,0.86}
\definecolor{cyan1}{rgb}{0.00,1.00,1.00}
\definecolor{cyan2}{rgb}{0.00,0.93,0.93}
\definecolor{cyan3}{rgb}{0.00,0.80,0.80}
\definecolor{cyan4}{rgb}{0.00,0.55,0.55}
\definecolor{cyan}{rgb}{0.00,1.00,1.00}
\definecolor{darkblue}{rgb}{0.00,0.00,0.55}
\definecolor{darkcyan}{rgb}{0.00,0.55,0.55}
\definecolor{darkgoldenrod}{rgb}{0.72,0.53,0.04}
\definecolor{darkgray}{rgb}{0.66,0.66,0.66}
\definecolor{darkgreen}{rgb}{0.00,0.39,0.00}
\definecolor{darkgrey}{rgb}{0.66,0.66,0.66}
\definecolor{darkkhaki}{rgb}{0.74,0.72,0.42}
\definecolor{darkmagenta}{rgb}{0.55,0.00,0.55}
\definecolor{darkolive}{rgb}{0.33,0.42,0.18}
\definecolor{darkorange}{rgb}{1.00,0.55,0.00}
\definecolor{darkorchid}{rgb}{0.60,0.20,0.80}
\definecolor{darkred}{rgb}{0.55,0.00,0.00}
\definecolor{darksalmon}{rgb}{0.91,0.59,0.48}
\definecolor{darksea}{rgb}{0.56,0.74,0.56}
\definecolor{darkslate}{rgb}{0.18,0.31,0.31}
\definecolor{darkslate}{rgb}{0.18,0.31,0.31}
\definecolor{darkslate}{rgb}{0.28,0.24,0.55}
\definecolor{darkturquoise}{rgb}{0.00,0.81,0.82}
\definecolor{darkviolet}{rgb}{0.58,0.00,0.83}
\definecolor{deeppink}{rgb}{1.00,0.08,0.58}
\definecolor{deepsky}{rgb}{0.00,0.75,1.00}
\definecolor{dimgray}{rgb}{0.41,0.41,0.41}
\definecolor{dimgrey}{rgb}{0.41,0.41,0.41}
\definecolor{dodgerblue}{rgb}{0.12,0.56,1.00}
\definecolor{firebrick1}{rgb}{1.00,0.19,0.19}
\definecolor{firebrick2}{rgb}{0.93,0.17,0.17}
\definecolor{firebrick3}{rgb}{0.80,0.15,0.15}
\definecolor{firebrick4}{rgb}{0.55,0.10,0.10}
\definecolor{firebrick}{rgb}{0.70,0.13,0.13}
\definecolor{floralwhite}{rgb}{1.00,0.98,0.94}
\definecolor{forestgreen}{rgb}{0.13,0.55,0.13}
\definecolor{gainsboro}{rgb}{0.86,0.86,0.86}
\definecolor{ghostwhite}{rgb}{0.97,0.97,1.00}
\definecolor{gold1}{rgb}{1.00,0.84,0.00}
\definecolor{gold2}{rgb}{0.93,0.79,0.00}
\definecolor{gold3}{rgb}{0.80,0.68,0.00}
\definecolor{gold4}{rgb}{0.55,0.46,0.00}
\definecolor{goldenrod1}{rgb}{1.00,0.76,0.15}
\definecolor{goldenrod2}{rgb}{0.93,0.71,0.13}
\definecolor{goldenrod3}{rgb}{0.80,0.61,0.11}
\definecolor{goldenrod4}{rgb}{0.55,0.41,0.08}
\definecolor{goldenrod}{rgb}{0.85,0.65,0.13}
\definecolor{gold}{rgb}{1.00,0.84,0.00}
\definecolor{gray0}{rgb}{0.00,0.00,0.00}
\definecolor{gray100}{rgb}{1.00,1.00,1.00}
\definecolor{gray10}{rgb}{0.10,0.10,0.10}
\definecolor{gray11}{rgb}{0.11,0.11,0.11}
\definecolor{gray12}{rgb}{0.12,0.12,0.12}
\definecolor{gray13}{rgb}{0.13,0.13,0.13}
\definecolor{gray14}{rgb}{0.14,0.14,0.14}
\definecolor{gray15}{rgb}{0.15,0.15,0.15}
\definecolor{gray16}{rgb}{0.16,0.16,0.16}
\definecolor{gray17}{rgb}{0.17,0.17,0.17}
\definecolor{gray18}{rgb}{0.18,0.18,0.18}
\definecolor{gray19}{rgb}{0.19,0.19,0.19}
\definecolor{gray1}{rgb}{0.01,0.01,0.01}
\definecolor{gray20}{rgb}{0.20,0.20,0.20}
\definecolor{gray21}{rgb}{0.21,0.21,0.21}
\definecolor{gray22}{rgb}{0.22,0.22,0.22}
\definecolor{gray23}{rgb}{0.23,0.23,0.23}
\definecolor{gray24}{rgb}{0.24,0.24,0.24}
\definecolor{gray25}{rgb}{0.25,0.25,0.25}
\definecolor{gray26}{rgb}{0.26,0.26,0.26}
\definecolor{gray27}{rgb}{0.27,0.27,0.27}
\definecolor{gray28}{rgb}{0.28,0.28,0.28}
\definecolor{gray29}{rgb}{0.29,0.29,0.29}
\definecolor{gray2}{rgb}{0.02,0.02,0.02}
\definecolor{gray30}{rgb}{0.30,0.30,0.30}
\definecolor{gray31}{rgb}{0.31,0.31,0.31}
\definecolor{gray32}{rgb}{0.32,0.32,0.32}
\definecolor{gray33}{rgb}{0.33,0.33,0.33}
\definecolor{gray34}{rgb}{0.34,0.34,0.34}
\definecolor{gray35}{rgb}{0.35,0.35,0.35}
\definecolor{gray36}{rgb}{0.36,0.36,0.36}
\definecolor{gray37}{rgb}{0.37,0.37,0.37}
\definecolor{gray38}{rgb}{0.38,0.38,0.38}
\definecolor{gray39}{rgb}{0.39,0.39,0.39}
\definecolor{gray3}{rgb}{0.03,0.03,0.03}
\definecolor{gray40}{rgb}{0.40,0.40,0.40}
\definecolor{gray41}{rgb}{0.41,0.41,0.41}
\definecolor{gray42}{rgb}{0.42,0.42,0.42}
\definecolor{gray43}{rgb}{0.43,0.43,0.43}
\definecolor{gray44}{rgb}{0.44,0.44,0.44}
\definecolor{gray45}{rgb}{0.45,0.45,0.45}
\definecolor{gray46}{rgb}{0.46,0.46,0.46}
\definecolor{gray47}{rgb}{0.47,0.47,0.47}
\definecolor{gray48}{rgb}{0.48,0.48,0.48}
\definecolor{gray49}{rgb}{0.49,0.49,0.49}
\definecolor{gray4}{rgb}{0.04,0.04,0.04}
\definecolor{gray50}{rgb}{0.50,0.50,0.50}
\definecolor{gray51}{rgb}{0.51,0.51,0.51}
\definecolor{gray52}{rgb}{0.52,0.52,0.52}
\definecolor{gray53}{rgb}{0.53,0.53,0.53}
\definecolor{gray54}{rgb}{0.54,0.54,0.54}
\definecolor{gray55}{rgb}{0.55,0.55,0.55}
\definecolor{gray56}{rgb}{0.56,0.56,0.56}
\definecolor{gray57}{rgb}{0.57,0.57,0.57}
\definecolor{gray58}{rgb}{0.58,0.58,0.58}
\definecolor{gray59}{rgb}{0.59,0.59,0.59}
\definecolor{gray5}{rgb}{0.05,0.05,0.05}
\definecolor{gray60}{rgb}{0.60,0.60,0.60}
\definecolor{gray61}{rgb}{0.61,0.61,0.61}
\definecolor{gray62}{rgb}{0.62,0.62,0.62}
\definecolor{gray63}{rgb}{0.63,0.63,0.63}
\definecolor{gray64}{rgb}{0.64,0.64,0.64}
\definecolor{gray65}{rgb}{0.65,0.65,0.65}
\definecolor{gray66}{rgb}{0.66,0.66,0.66}
\definecolor{gray67}{rgb}{0.67,0.67,0.67}
\definecolor{gray68}{rgb}{0.68,0.68,0.68}
\definecolor{gray69}{rgb}{0.69,0.69,0.69}
\definecolor{gray6}{rgb}{0.06,0.06,0.06}
\definecolor{gray70}{rgb}{0.70,0.70,0.70}
\definecolor{gray71}{rgb}{0.71,0.71,0.71}
\definecolor{gray72}{rgb}{0.72,0.72,0.72}
\definecolor{gray73}{rgb}{0.73,0.73,0.73}
\definecolor{gray74}{rgb}{0.74,0.74,0.74}
\definecolor{gray75}{rgb}{0.75,0.75,0.75}
\definecolor{gray76}{rgb}{0.76,0.76,0.76}
\definecolor{gray77}{rgb}{0.77,0.77,0.77}
\definecolor{gray78}{rgb}{0.78,0.78,0.78}
\definecolor{gray79}{rgb}{0.79,0.79,0.79}
\definecolor{gray7}{rgb}{0.07,0.07,0.07}
\definecolor{gray80}{rgb}{0.80,0.80,0.80}
\definecolor{gray81}{rgb}{0.81,0.81,0.81}
\definecolor{gray82}{rgb}{0.82,0.82,0.82}
\definecolor{gray83}{rgb}{0.83,0.83,0.83}
\definecolor{gray84}{rgb}{0.84,0.84,0.84}
\definecolor{gray85}{rgb}{0.85,0.85,0.85}
\definecolor{gray86}{rgb}{0.86,0.86,0.86}
\definecolor{gray87}{rgb}{0.87,0.87,0.87}
\definecolor{gray88}{rgb}{0.88,0.88,0.88}
\definecolor{gray89}{rgb}{0.89,0.89,0.89}
\definecolor{gray8}{rgb}{0.08,0.08,0.08}
\definecolor{gray90}{rgb}{0.90,0.90,0.90}
\definecolor{gray91}{rgb}{0.91,0.91,0.91}
\definecolor{gray92}{rgb}{0.92,0.92,0.92}
\definecolor{gray93}{rgb}{0.93,0.93,0.93}
\definecolor{gray94}{rgb}{0.94,0.94,0.94}
\definecolor{gray95}{rgb}{0.95,0.95,0.95}
\definecolor{gray96}{rgb}{0.96,0.96,0.96}
\definecolor{gray97}{rgb}{0.97,0.97,0.97}
\definecolor{gray98}{rgb}{0.98,0.98,0.98}
\definecolor{gray99}{rgb}{0.99,0.99,0.99}
\definecolor{gray9}{rgb}{0.09,0.09,0.09}
\definecolor{gray}{rgb}{0.75,0.75,0.75}
\definecolor{green1}{rgb}{0.00,1.00,0.00}
\definecolor{green2}{rgb}{0.00,0.93,0.00}
\definecolor{green3}{rgb}{0.00,0.80,0.00}
\definecolor{green4}{rgb}{0.00,0.55,0.00}
\definecolor{greenyellow}{rgb}{0.68,1.00,0.18}
\definecolor{green}{rgb}{0.00,1.00,0.00}
\definecolor{grey0}{rgb}{0.00,0.00,0.00}
\definecolor{grey100}{rgb}{1.00,1.00,1.00}
\definecolor{grey10}{rgb}{0.10,0.10,0.10}
\definecolor{grey11}{rgb}{0.11,0.11,0.11}
\definecolor{grey12}{rgb}{0.12,0.12,0.12}
\definecolor{grey13}{rgb}{0.13,0.13,0.13}
\definecolor{grey14}{rgb}{0.14,0.14,0.14}
\definecolor{grey15}{rgb}{0.15,0.15,0.15}
\definecolor{grey16}{rgb}{0.16,0.16,0.16}
\definecolor{grey17}{rgb}{0.17,0.17,0.17}
\definecolor{grey18}{rgb}{0.18,0.18,0.18}
\definecolor{grey19}{rgb}{0.19,0.19,0.19}
\definecolor{grey1}{rgb}{0.01,0.01,0.01}
\definecolor{grey20}{rgb}{0.20,0.20,0.20}
\definecolor{grey21}{rgb}{0.21,0.21,0.21}
\definecolor{grey22}{rgb}{0.22,0.22,0.22}
\definecolor{grey23}{rgb}{0.23,0.23,0.23}
\definecolor{grey24}{rgb}{0.24,0.24,0.24}
\definecolor{grey25}{rgb}{0.25,0.25,0.25}
\definecolor{grey26}{rgb}{0.26,0.26,0.26}
\definecolor{grey27}{rgb}{0.27,0.27,0.27}
\definecolor{grey28}{rgb}{0.28,0.28,0.28}
\definecolor{grey29}{rgb}{0.29,0.29,0.29}
\definecolor{grey2}{rgb}{0.02,0.02,0.02}
\definecolor{grey30}{rgb}{0.30,0.30,0.30}
\definecolor{grey31}{rgb}{0.31,0.31,0.31}
\definecolor{grey32}{rgb}{0.32,0.32,0.32}
\definecolor{grey33}{rgb}{0.33,0.33,0.33}
\definecolor{grey34}{rgb}{0.34,0.34,0.34}
\definecolor{grey35}{rgb}{0.35,0.35,0.35}
\definecolor{grey36}{rgb}{0.36,0.36,0.36}
\definecolor{grey37}{rgb}{0.37,0.37,0.37}
\definecolor{grey38}{rgb}{0.38,0.38,0.38}
\definecolor{grey39}{rgb}{0.39,0.39,0.39}
\definecolor{grey3}{rgb}{0.03,0.03,0.03}
\definecolor{grey40}{rgb}{0.40,0.40,0.40}
\definecolor{grey41}{rgb}{0.41,0.41,0.41}
\definecolor{grey42}{rgb}{0.42,0.42,0.42}
\definecolor{grey43}{rgb}{0.43,0.43,0.43}
\definecolor{grey44}{rgb}{0.44,0.44,0.44}
\definecolor{grey45}{rgb}{0.45,0.45,0.45}
\definecolor{grey46}{rgb}{0.46,0.46,0.46}
\definecolor{grey47}{rgb}{0.47,0.47,0.47}
\definecolor{grey48}{rgb}{0.48,0.48,0.48}
\definecolor{grey49}{rgb}{0.49,0.49,0.49}
\definecolor{grey4}{rgb}{0.04,0.04,0.04}
\definecolor{grey50}{rgb}{0.50,0.50,0.50}
\definecolor{grey51}{rgb}{0.51,0.51,0.51}
\definecolor{grey52}{rgb}{0.52,0.52,0.52}
\definecolor{grey53}{rgb}{0.53,0.53,0.53}
\definecolor{grey54}{rgb}{0.54,0.54,0.54}
\definecolor{grey55}{rgb}{0.55,0.55,0.55}
\definecolor{grey56}{rgb}{0.56,0.56,0.56}
\definecolor{grey57}{rgb}{0.57,0.57,0.57}
\definecolor{grey58}{rgb}{0.58,0.58,0.58}
\definecolor{grey59}{rgb}{0.59,0.59,0.59}
\definecolor{grey5}{rgb}{0.05,0.05,0.05}
\definecolor{grey60}{rgb}{0.60,0.60,0.60}
\definecolor{grey61}{rgb}{0.61,0.61,0.61}
\definecolor{grey62}{rgb}{0.62,0.62,0.62}
\definecolor{grey63}{rgb}{0.63,0.63,0.63}
\definecolor{grey64}{rgb}{0.64,0.64,0.64}
\definecolor{grey65}{rgb}{0.65,0.65,0.65}
\definecolor{grey66}{rgb}{0.66,0.66,0.66}
\definecolor{grey67}{rgb}{0.67,0.67,0.67}
\definecolor{grey68}{rgb}{0.68,0.68,0.68}
\definecolor{grey69}{rgb}{0.69,0.69,0.69}
\definecolor{grey6}{rgb}{0.06,0.06,0.06}
\definecolor{grey70}{rgb}{0.70,0.70,0.70}
\definecolor{grey71}{rgb}{0.71,0.71,0.71}
\definecolor{grey72}{rgb}{0.72,0.72,0.72}
\definecolor{grey73}{rgb}{0.73,0.73,0.73}
\definecolor{grey74}{rgb}{0.74,0.74,0.74}
\definecolor{grey75}{rgb}{0.75,0.75,0.75}
\definecolor{grey76}{rgb}{0.76,0.76,0.76}
\definecolor{grey77}{rgb}{0.77,0.77,0.77}
\definecolor{grey78}{rgb}{0.78,0.78,0.78}
\definecolor{grey79}{rgb}{0.79,0.79,0.79}
\definecolor{grey7}{rgb}{0.07,0.07,0.07}
\definecolor{grey80}{rgb}{0.80,0.80,0.80}
\definecolor{grey81}{rgb}{0.81,0.81,0.81}
\definecolor{grey82}{rgb}{0.82,0.82,0.82}
\definecolor{grey83}{rgb}{0.83,0.83,0.83}
\definecolor{grey84}{rgb}{0.84,0.84,0.84}
\definecolor{grey85}{rgb}{0.85,0.85,0.85}
\definecolor{grey86}{rgb}{0.86,0.86,0.86}
\definecolor{grey87}{rgb}{0.87,0.87,0.87}
\definecolor{grey88}{rgb}{0.88,0.88,0.88}
\definecolor{grey89}{rgb}{0.89,0.89,0.89}
\definecolor{grey8}{rgb}{0.08,0.08,0.08}
\definecolor{grey90}{rgb}{0.90,0.90,0.90}
\definecolor{grey91}{rgb}{0.91,0.91,0.91}
\definecolor{grey92}{rgb}{0.92,0.92,0.92}
\definecolor{grey93}{rgb}{0.93,0.93,0.93}
\definecolor{grey94}{rgb}{0.94,0.94,0.94}
\definecolor{grey95}{rgb}{0.95,0.95,0.95}
\definecolor{grey96}{rgb}{0.96,0.96,0.96}
\definecolor{grey97}{rgb}{0.97,0.97,0.97}
\definecolor{grey98}{rgb}{0.98,0.98,0.98}
\definecolor{grey99}{rgb}{0.99,0.99,0.99}
\definecolor{grey9}{rgb}{0.09,0.09,0.09}
\definecolor{grey}{rgb}{0.75,0.75,0.75}
\definecolor{honeydew1}{rgb}{0.94,1.00,0.94}
\definecolor{honeydew2}{rgb}{0.88,0.93,0.88}
\definecolor{honeydew3}{rgb}{0.76,0.80,0.76}
\definecolor{honeydew4}{rgb}{0.51,0.55,0.51}
\definecolor{honeydew}{rgb}{0.94,1.00,0.94}
\definecolor{hotpink}{rgb}{1.00,0.41,0.71}
\definecolor{indianred}{rgb}{0.80,0.36,0.36}
\definecolor{ivory1}{rgb}{1.00,1.00,0.94}
\definecolor{ivory2}{rgb}{0.93,0.93,0.88}
\definecolor{ivory3}{rgb}{0.80,0.80,0.76}
\definecolor{ivory4}{rgb}{0.55,0.55,0.51}
\definecolor{ivory}{rgb}{1.00,1.00,0.94}
\definecolor{khaki1}{rgb}{1.00,0.96,0.56}
\definecolor{khaki2}{rgb}{0.93,0.90,0.52}
\definecolor{khaki3}{rgb}{0.80,0.78,0.45}
\definecolor{khaki4}{rgb}{0.55,0.53,0.31}
\definecolor{khaki}{rgb}{0.94,0.90,0.55}
\definecolor{lavenderblush}{rgb}{1.00,0.94,0.96}
\definecolor{lavender}{rgb}{0.90,0.90,0.98}
\definecolor{lawngreen}{rgb}{0.49,0.99,0.00}
\definecolor{lemonchiffon}{rgb}{1.00,0.98,0.80}
\definecolor{lightblue}{rgb}{0.68,0.85,0.90}
\definecolor{lightcoral}{rgb}{0.94,0.50,0.50}
\definecolor{lightcyan}{rgb}{0.88,1.00,1.00}
\definecolor{lightgoldenrod}{rgb}{0.93,0.87,0.51}
\definecolor{lightgoldenrod}{rgb}{0.98,0.98,0.82}
\definecolor{lightgray}{rgb}{0.83,0.83,0.83}
\definecolor{lightgreen}{rgb}{0.56,0.93,0.56}
\definecolor{lightgrey}{rgb}{0.83,0.83,0.83}
\definecolor{lightpink}{rgb}{1.00,0.71,0.76}
\definecolor{lightsalmon}{rgb}{1.00,0.63,0.48}
\definecolor{lightsea}{rgb}{0.13,0.70,0.67}
\definecolor{lightsky}{rgb}{0.53,0.81,0.98}
\definecolor{lightslate}{rgb}{0.47,0.53,0.60}
\definecolor{lightslate}{rgb}{0.47,0.53,0.60}
\definecolor{lightslate}{rgb}{0.52,0.44,1.00}
\definecolor{lightsteel}{rgb}{0.69,0.77,0.87}
\definecolor{lightyellow}{rgb}{1.00,1.00,0.88}
\definecolor{limegreen}{rgb}{0.20,0.80,0.20}
\definecolor{linen}{rgb}{0.98,0.94,0.90}
\definecolor{magenta1}{rgb}{1.00,0.00,1.00}
\definecolor{magenta2}{rgb}{0.93,0.00,0.93}
\definecolor{magenta3}{rgb}{0.80,0.00,0.80}
\definecolor{magenta4}{rgb}{0.55,0.00,0.55}
\definecolor{magenta}{rgb}{1.00,0.00,1.00}
\definecolor{maroon1}{rgb}{1.00,0.20,0.70}
\definecolor{maroon2}{rgb}{0.93,0.19,0.65}
\definecolor{maroon3}{rgb}{0.80,0.16,0.56}
\definecolor{maroon4}{rgb}{0.55,0.11,0.38}
\definecolor{maroon}{rgb}{0.69,0.19,0.38}
\definecolor{mediumaquamarine}{rgb}{0.40,0.80,0.67}
\definecolor{mediumblue}{rgb}{0.00,0.00,0.80}
\definecolor{mediumorchid}{rgb}{0.73,0.33,0.83}
\definecolor{mediumpurple}{rgb}{0.58,0.44,0.86}
\definecolor{mediumsea}{rgb}{0.24,0.70,0.44}
\definecolor{mediumslate}{rgb}{0.48,0.41,0.93}
\definecolor{mediumspring}{rgb}{0.00,0.98,0.60}
\definecolor{mediumturquoise}{rgb}{0.28,0.82,0.80}
\definecolor{mediumviolet}{rgb}{0.78,0.08,0.52}
\definecolor{midnightblue}{rgb}{0.10,0.10,0.44}
\definecolor{mintcream}{rgb}{0.96,1.00,0.98}
\definecolor{mistyrose}{rgb}{1.00,0.89,0.88}
\definecolor{moccasin}{rgb}{1.00,0.89,0.71}
\definecolor{navajowhite}{rgb}{1.00,0.87,0.68}
\definecolor{navyblue}{rgb}{0.00,0.00,0.50}
\definecolor{navy}{rgb}{0.00,0.00,0.50}
\definecolor{oldlace}{rgb}{0.99,0.96,0.90}
\definecolor{olivedrab}{rgb}{0.42,0.56,0.14}
\definecolor{orange1}{rgb}{1.00,0.65,0.00}
\definecolor{orange2}{rgb}{0.93,0.60,0.00}
\definecolor{orange3}{rgb}{0.80,0.52,0.00}
\definecolor{orange4}{rgb}{0.55,0.35,0.00}
\definecolor{orangered}{rgb}{1.00,0.27,0.00}
\definecolor{orange}{rgb}{1.00,0.65,0.00}
\definecolor{orchid1}{rgb}{1.00,0.51,0.98}
\definecolor{orchid2}{rgb}{0.93,0.48,0.91}
\definecolor{orchid3}{rgb}{0.80,0.41,0.79}
\definecolor{orchid4}{rgb}{0.55,0.28,0.54}
\definecolor{orchid}{rgb}{0.85,0.44,0.84}
\definecolor{palegoldenrod}{rgb}{0.93,0.91,0.67}
\definecolor{palegreen}{rgb}{0.60,0.98,0.60}
\definecolor{paleturquoise}{rgb}{0.69,0.93,0.93}
\definecolor{paleviolet}{rgb}{0.86,0.44,0.58}
\definecolor{papayawhip}{rgb}{1.00,0.94,0.84}
\definecolor{peachpuff}{rgb}{1.00,0.85,0.73}
\definecolor{peru}{rgb}{0.80,0.52,0.25}
\definecolor{pink1}{rgb}{1.00,0.71,0.77}
\definecolor{pink2}{rgb}{0.93,0.66,0.72}
\definecolor{pink3}{rgb}{0.80,0.57,0.62}
\definecolor{pink4}{rgb}{0.55,0.39,0.42}
\definecolor{pink}{rgb}{1.00,0.75,0.80}
\definecolor{plum1}{rgb}{1.00,0.73,1.00}
\definecolor{plum2}{rgb}{0.93,0.68,0.93}
\definecolor{plum3}{rgb}{0.80,0.59,0.80}
\definecolor{plum4}{rgb}{0.55,0.40,0.55}
\definecolor{plum}{rgb}{0.87,0.63,0.87}
\definecolor{powderblue}{rgb}{0.69,0.88,0.90}
\definecolor{purple1}{rgb}{0.61,0.19,1.00}
\definecolor{purple2}{rgb}{0.57,0.17,0.93}
\definecolor{purple3}{rgb}{0.49,0.15,0.80}
\definecolor{purple4}{rgb}{0.33,0.10,0.55}
\definecolor{purple}{rgb}{0.63,0.13,0.94}
\definecolor{red1}{rgb}{1.00,0.00,0.00}
\definecolor{red2}{rgb}{0.93,0.00,0.00}
\definecolor{red3}{rgb}{0.80,0.00,0.00}
\definecolor{red4}{rgb}{0.55,0.00,0.00}
\definecolor{red}{rgb}{1.00,0.00,0.00}
\definecolor{rosybrown}{rgb}{0.74,0.56,0.56}
\definecolor{royalblue}{rgb}{0.25,0.41,0.88}
\definecolor{saddlebrown}{rgb}{0.55,0.27,0.07}
\definecolor{salmon1}{rgb}{1.00,0.55,0.41}
\definecolor{salmon2}{rgb}{0.93,0.51,0.38}
\definecolor{salmon3}{rgb}{0.80,0.44,0.33}
\definecolor{salmon4}{rgb}{0.55,0.30,0.22}
\definecolor{salmon}{rgb}{0.98,0.50,0.45}
\definecolor{sandybrown}{rgb}{0.96,0.64,0.38}
\definecolor{seagreen}{rgb}{0.18,0.55,0.34}
\definecolor{seashell1}{rgb}{1.00,0.96,0.93}
\definecolor{seashell2}{rgb}{0.93,0.90,0.87}
\definecolor{seashell3}{rgb}{0.80,0.77,0.75}
\definecolor{seashell4}{rgb}{0.55,0.53,0.51}
\definecolor{seashell}{rgb}{1.00,0.96,0.93}
\definecolor{sienna1}{rgb}{1.00,0.51,0.28}
\definecolor{sienna2}{rgb}{0.93,0.47,0.26}
\definecolor{sienna3}{rgb}{0.80,0.41,0.22}
\definecolor{sienna4}{rgb}{0.55,0.28,0.15}
\definecolor{sienna}{rgb}{0.63,0.32,0.18}
\definecolor{skyblue}{rgb}{0.53,0.81,0.92}
\definecolor{slateblue}{rgb}{0.42,0.35,0.80}
\definecolor{slategray}{rgb}{0.44,0.50,0.56}
\definecolor{slategrey}{rgb}{0.44,0.50,0.56}
\definecolor{snow1}{rgb}{1.00,0.98,0.98}
\definecolor{snow2}{rgb}{0.93,0.91,0.91}
\definecolor{snow3}{rgb}{0.80,0.79,0.79}
\definecolor{snow4}{rgb}{0.55,0.54,0.54}
\definecolor{snow}{rgb}{1.00,0.98,0.98}
\definecolor{springgreen}{rgb}{0.00,1.00,0.50}
\definecolor{steelblue}{rgb}{0.27,0.51,0.71}
\definecolor{tan1}{rgb}{1.00,0.65,0.31}
\definecolor{tan2}{rgb}{0.93,0.60,0.29}
\definecolor{tan3}{rgb}{0.80,0.52,0.25}
\definecolor{tan4}{rgb}{0.55,0.35,0.17}
\definecolor{tan}{rgb}{0.82,0.71,0.55}
\definecolor{thistle1}{rgb}{1.00,0.88,1.00}
\definecolor{thistle2}{rgb}{0.93,0.82,0.93}
\definecolor{thistle3}{rgb}{0.80,0.71,0.80}
\definecolor{thistle4}{rgb}{0.55,0.48,0.55}
\definecolor{thistle}{rgb}{0.85,0.75,0.85}
\definecolor{tomato1}{rgb}{1.00,0.39,0.28}
\definecolor{tomato2}{rgb}{0.93,0.36,0.26}
\definecolor{tomato3}{rgb}{0.80,0.31,0.22}
\definecolor{tomato4}{rgb}{0.55,0.21,0.15}
\definecolor{tomato}{rgb}{1.00,0.39,0.28}
\definecolor{turquoise1}{rgb}{0.00,0.96,1.00}
\definecolor{turquoise2}{rgb}{0.00,0.90,0.93}
\definecolor{turquoise3}{rgb}{0.00,0.77,0.80}
\definecolor{turquoise4}{rgb}{0.00,0.53,0.55}
\definecolor{turquoise}{rgb}{0.25,0.88,0.82}
\definecolor{violetred}{rgb}{0.82,0.13,0.56}
\definecolor{violet}{rgb}{0.93,0.51,0.93}
\definecolor{wheat1}{rgb}{1.00,0.91,0.73}
\definecolor{wheat2}{rgb}{0.93,0.85,0.68}
\definecolor{wheat3}{rgb}{0.80,0.73,0.59}
\definecolor{wheat4}{rgb}{0.55,0.49,0.40}
\definecolor{wheat}{rgb}{0.96,0.87,0.70}
\definecolor{whitesmoke}{rgb}{0.96,0.96,0.96}
\definecolor{white}{rgb}{1.00,1.00,1.00}
\definecolor{yellow1}{rgb}{1.00,1.00,0.00}
\definecolor{yellow2}{rgb}{0.93,0.93,0.00}
\definecolor{yellow3}{rgb}{0.80,0.80,0.00}
\definecolor{yellow4}{rgb}{0.55,0.55,0.00}
\definecolor{yellowgreen}{rgb}{0.60,0.80,0.20}
\definecolor{yellow}{rgb}{1.00,1.00,0.00}

\usepackage[colorinlistoftodos]{todonotes}

\numberwithin{equation}{section}
\newtheorem{thm}{Theorem}[section]
\newtheorem{cor}[thm]{Corollary}
\newtheorem{lem}[thm]{Lemma}
\newtheorem{prop}[thm]{Proposition}

\theoremstyle{definition}
\newtheorem{defn}[thm]{Definition}

\newcommand{\lxr}{\longrightarrow}

\newcommand{\mr}{\mathsf{r}}
\newcommand{\mq}{\mathsf{q}}
\newcommand{\mi}{\mathsf{i}}

\newcommand{\ml}{\mathsf{l}}
\newcommand{\me}{\mathsf{e}}

\newcommand{\map}{\mathsf{p}}

 \DeclareMathOperator{\inc}{\mathsf{inc}}

\DeclareMathOperator{\Hom}{\mathsf{Hom}}

\newsavebox{\proofbox}
\savebox{\proofbox}{\begin{picture}(7,7)%
  \put(0,0){\framebox(7,7){}}\end{picture}}

\begin{document}

\title[]{(Gorenstein) silting modules in recollements}
\author[]{Nan Gao$^{*}$,  Jing Ma
}
\address{Department of Mathematics, Shanghai University, Shanghai 200444, PR China}
\thanks{* is the corresponding author.}
\email{nangao@shu.edu.cn,  majingmmmm@shu.edu.cn}
\thanks{Supported by the National Natural Science Foundation of China (Grant No. 11771272 and 11871326).}

\date{\today}


\subjclass[2010]{}

\begin{abstract} \ In the paper, we focus on the silting properties and the combinatorial properties of silting and Gorenstein, which is called Gorenstein silting, where the main tools used are recollements of module categories and tensor products. For a ring $A$ and its idempotent ideal $J$, we show that an $A/J\mbox{-}$module $T$ is a silting $A$-module if and only if $T$ is a silting $A/J\mbox{-}$module. For the finite dimensional $k\mbox{-}$algebras, with $k$ a field, we show that the tensor products of silting modules are still silting. We also show that the (partial) Gorenstein silting properties can be glued by the recollements of module categories of Noetherian rings. As a consequence, we glue the Gorenstein silting modules of an upper triangular matrix Gorenstein ring by those of the involved rings.
\end{abstract}

\maketitle


\vskip 20pt

\section{\bf Introduction}

\vskip 10pt

Silting modules over an arbitrary ring were introduced by Angeleri H\"{u}gel-Marks-Vit\'{o}ria \cite{AMV1}, which are intended to generalize tilting modules. In particular, silting modules coincide with support $\tau$-tilting modules, which were introduced by Adachi-Iyama-Reiten \cite{AIR}, in the category of finitely generated modules of finite dimensional algebras. In \cite{AMV1}, it is also proved how silting modules relate with 2-term silting complexes. Silting complexes were first introduced by Keller and Vossieck \cite{KV} to study $t\mbox{-}$structures in the bounded derived category of representations of Dynkin quivers. Silting modules share many properties with tilting modules and support $\tau$-tilting modules, so it has been studied and concerned by many scholars. The subject has been developed to an advanced level, see for examples(\cite{MS, AMV2, A, BZ1, BZ2, HKM, MM, KY}).

\vskip 10pt

The study of Gorenstein homological algebra is due to Enochs and Jenda \cite{EJ1}. They introduced the concept of Gorenstein-projective modules, which are as a generalization of finitely generated modules of G-dimension zero over a two-sided Noetherian ring, in the sense of Auslander and Bridger \cite{AB}. The main idea of Gorenstein homological algebra is to replace projective modules by Gorenstein-projective modules, which is useful to study some Gorenstein properties. Gao-Ma-Zhang \cite{GMZ} studied Gorenstein silting modules, intending to understand silting modules more comprehensively and fully in the Gorenstein homological algebra, and also to fuse the properties of silting modules and Gorenstein-projective modules. Although the usual silting modules have been used effectively in Gorenstein homological algebra, Gorenstein silting  modules has some advantages in the relative setting. For examples, the Gorenstein silting module is the module-theoretic counterpart of the 2-term Gorenstein silting complex (\cite{CW, GMZ}), which generates the bounded homotopy category of Gorenstein-projective modules.

\vskip 10pt

Based on these works, we aim to consider fundamental problems that how the silting (resp. Gorenstein silting) property transfers and glues in the recollements of module categories, and that how the silting property is preserved under the tensor product.

\vskip 10pt

For module categories of rings, each idempotent $e$ in a ring $A$ provides natural analogues of Grothendieck's six functors, defining recollements of module categories
\[
\xymatrix@C=0.5cm{
{A/AeA}\mbox{-}{\rm Mod} \ar[rrr]^{\inc} &&& {A}\mbox{-}{\rm Mod} \ar[rrr]^{e(-)} \ar
@/_1.5pc/[lll]_{A/AeA\otimes_A-}  \ar
 @/^1.5pc/[lll]^{\Hom_A(A/AeA,-)} &&& {eAe}\mbox{-}{\rm Mod}.
\ar @/_1.5pc/[lll]_{Ae\otimes_{eAe}-} \ar
 @/^1.5pc/[lll]^{\Hom_{eAe}(eA,-)}
 }
\]
These, and more generally recollements of abelian
categories introduced by Beilinson-Bernstein-Deligne \cite{BBD} have been used in various contexts (see for instance \cite{Buchweitz, Pira, Psaroud:survey, FZ, GKP}).

\vskip 10pt

In the paper, we give the answers by the tools of recollements and tensor products. Our main results are the follows:

\vskip 5pt

{\bf Theorem }\  The following statements hold.

\vskip 5pt

 (Theorem~\ref{cosiltingstratifying}) \ Let $A$ be a ring and $J$ the idempotent ideal of $A$. Let $T$ be a left $A/J\mbox{-}$module. Then $T$ is a silting $A$-module  if and only if $T$  is a silting $A/J$-module.

\vskip 5pt

(Theorem~\ref{tpofsilting})\ Let $A$ and $B$ be two finite dimensional $k$-algebras over a field $k$. Suppose that $T$ is a silting $A\mbox{-}$module and $S$ is a silting $B\mbox{-}$module. Then $T\otimes_{k} S$ is a silting $A\otimes_{k} B\mbox{-}$module.

\vskip 5pt

(Theorem~\ref{Gsilting})\ Let $\Gamma=\begin{pmatrix}\begin{smallmatrix}
A & N \\
0 & B \\
\end{smallmatrix}\end{pmatrix}$ be a Gorenstein ring such that ${\rm gl.dim}A<\infty$ and $_{A}N$ and $N_{B}$ are projective.
Let $X$ be an $A\mbox{-}$module and $Y$ a $B\mbox{-}$module. Then the following are equivalent:
\vskip 5pt
\begin{enumerate}
\item {\rm (a)}\ $(X, 0)\oplus (N\otimes_{B}Y, Y)$ is a Gorenstein silting $\Gamma\mbox{-}$module;

\vskip 5pt

\noindent{\rm (b)}\ there exists a $G\mbox{-}$exact sequence $(P, 0)\oplus (N\otimes_{B}E, E)\stackrel{\lambda}{\lxr} (X_{0}, 0)\oplus (N\otimes_{B}Y_{0}, Y_{0})\lxr (X_{-1}, 0)\oplus (N\otimes_{B}Y_{-1}, Y_{-1})\lxr 0$ with $\lambda=\begin{pmatrix}\begin{smallmatrix}
(\phi, 0) & 0 \\
0 & ({\rm Id}_{N}\otimes \psi, \psi) \\
\end{smallmatrix}\end{pmatrix}$, $X_{i}\in {\rm Add}X$ and $Y_{i}\in {\rm Add}Y$ such that $\lambda$ is the left $D_{\theta}\mbox{-}$approximation, $i=-1,0$, for every Gorenstein-projective $\Gamma\mbox{-}$module $(P, 0)\oplus (N\otimes_{B}E, E)$.

\vskip 5pt

\item {\rm (c)}\ $X$ is a Gorenstein silting $A\mbox{-}$module and $N\otimes_{B}Y\in {\rm Gen}_{G}X$;

\vskip 5pt

\noindent{\rm (d)}\ $Y$ is a Gorenstein silting $B\mbox{-}$module;

\vskip 5pt

\noindent{\rm (e)}\ there exists an exact sequence $P\stackrel{\phi}{\lxr} X_{0}\lxr X_{-1}\lxr 0$ with $X_{0}$ and $X_{-1}$ in ${\rm Add}X$ such that $\phi$ is the left $D_{\theta_{X}}\mbox{-}$approximation, for each $P\in A\mbox{-}{\rm P}$;

\vskip 5pt

\noindent{\rm (f)}\ there exists a $G\mbox{-}$exact sequence $E\stackrel{\psi}{\lxr} Y_{0}\lxr Y_{-1}\lxr 0$ with $Y_{0}$ and $Y_{-1}$ in ${\rm Add}Y$ such that $\psi$ is the left $D_{\theta_{Y}}\mbox{-}$approximation, for each $E\in B\mbox{-}{\rm GP}$.
\end{enumerate}

\vskip 20pt

\section{\bf Preliminaries}

\vskip 10pt

In this section, we recall some basic definitions and facts that will be used throughout the paper.

\vskip 10pt

Let $A$ be an associative ring and $A\mbox{-}{\rm Mod}$  the category of  left modules.
Denote by $A\mbox{-}{\rm P}$ the full subcategories of $A\mbox{-}{\rm Mod}$ consisting of projective modules.
A module $G$ of $A\mbox{-}{\rm Mod}$ is Gorenstein-projective if there is an exact sequence
$$\cdots \lxr P^{-1}\lxr P^{0}\stackrel{d^{0}}{\lxr} P^{1}\lxr \cdots$$
of projective modules of $A\mbox{-}{\rm Mod}$, which stays exact after applying ${\rm Hom}_{A}(-, P)$ for each projective module $P$, such that $G\cong {\rm Ker}d^{0}$ (see \cite{EJ1, EJ2}). Denote by $A\mbox{-}{\rm GP}$  the full subcategories of $A\mbox{-}{\rm Mod}$ consisting of Gorenstein-projective modules.

\vskip 10pt

Given subcategories $\mathcal{X},\ \mathcal{Y}$ of $A\mbox{-}{\rm Mod}$.  Following~\cite{AS}, a morphism $f:X\lxr M$ with $X\in \mathcal{X}$ is called a right $\mathcal{X}\mbox{-}$approximation of $M$ in $A\mbox{-}{\rm Mod}$ if any morphism from a module in $\mathcal{X}$ to $M$ factors through $f$. $\mathcal{X}$ is called contravariantly finite if any module in $A\mbox{-}{\rm Mod}$ admits a right $\mathcal{X}\mbox{-}$approximation.

\vskip 10pt

Following~\cite{EJ2}, an exact sequence $G_{1}\xrightarrow{d^{1}} G_{0}\lxr M\lxr 0 \ (*)$ is called a proper Gorenstein-projective presentation of $M$ if each $G_{i}$ is Gorenstein-projective and ${\rm Hom}_{A}(G,$ $G_{1})\lxr {\rm Hom}_{A}(G,G_{0}) \lxr {\rm Hom}_{A}(G,T)\lxr 0$ is exact for any Gorenstein-projective module $G$. $(*)$ is minimal if $G_{1}\lxr {\rm Im}d^{1}$ and $G_{0}\lxr M$ are right $A\mbox{-}{\rm GP}\mbox{-}$approximation.  Moreover, the exact sequence $0\lxr G_{n}\lxr \ldots \lxr G_{1} \lxr G_{0} \rightarrow M \rightarrow 0$  is called a proper Gorenstein-projective resolution of $M$ of length $n$ for some non-negative integer $n$, if each $G_{i}$ is all Gorenstein-projective and $0\lxr {\rm Hom}_{A}(G, G_{n})\lxr \cdots \lxr {\rm Hom}_{A}(G,G_{0}) \lxr {\rm Hom}_{A}(G,M)\lxr 0$ is exact for any Gorenstein-projective module $G$.

\vskip 10pt

\subsection{\bf Triangular matrix rings}\ Let $A$ and $B$ be rings, $_{A}N_{B}$ an $(A, B)\mbox{-}$bimodule. Then the triangular matrix ring
$$\Gamma:=\begin{pmatrix}
A & N \\
0 & B \\
\end{pmatrix}$$
can be defined by the ordinary operations on matrices.

\vskip 5pt

Recall from \cite{ARS,Gr} that a left $\Gamma\mbox{-}$module can be identified with a triple $(X, Y, f)$, or simply $(X, Y)$ if $f$ is clear, where $X\in A\mbox{-}{\rm Mod},\ Y\in B\mbox{-}{\rm Mod}$, and $f: N\otimes_{B}Y\lxr X$ is an $A\mbox{-}$map. A $\Gamma\mbox{-}$map $(X, Y, f)\lxr (X', Y', f')$ will be identified with a pair $(a, b)$, where $a\in {\rm Hom}_{A}(X, X'),\ b\in {\rm Hom}_{B}(Y, Y')$, such that the following diagram commutes:
\[
\xymatrix@C=50pt{
N\otimes_{B}Y \ar[r]^{f}\ar[d]^{{\rm Id}_{N}\otimes b}  &  X \ar[d]^{a}    \\
N\otimes_{B}Y' \ar[r]^{f'}  &  X'.      }
\]
A sequence $0\lxr (X_{1}, Y_{1}, f_{1})\lxr (X_{2}, Y_{2}, f_{2})\lxr (X_{3}, Y_{3}, f_{3})\lxr 0$ in $\Gamma\mbox{-}{\rm Mod}$ is exact if and only if $0\lxr X_{1}\lxr X_{2}\lxr X_{3}\lxr 0$ and $0\lxr Y_{1}\lxr Y_{2}\lxr Y_{3}\lxr 0$ are exact. Let $\Gamma$ be a Noetherian ring. Indecomposable projective $\Gamma\mbox{-}$modules are exactly $(P, 0)$ and $(N\otimes_{B}Q, Q)$, where $P$ runs over indecomposable projective $A\mbox{-}$modules, and $Q$ runs over indecomposable projective $B\mbox{-}$modules.

\vskip 10pt

\subsection{\bf Silting modules}\ Let $A$ be an associative ring.  We  denote by ${\rm Add}T$  the class of all modules which are isomorphic to direct summands of direct sums  of copies of $T$, for an $A\mbox{-}$module $T$.

\vskip 10pt

Let $\sigma: P_{1}\lxr P_{0}$ be a morphism of projective $A\mbox{-}$modules. Consider the class:
$$D_{\sigma}:=\{X\in A\mbox{-}{\rm Mod}\mid {\rm Hom}_{A}(\sigma, X) \ is \ an \ epimorphism\}.$$

\vskip 5pt

\begin{defn}\ (\cite[Definition 3.7]{AMV1}) \ We call  that an $A$-module $X$ is

\vskip 5pt

$\bullet$ \ partial silting if there is a projective presentation $\sigma$ of $X$ such that

\vskip 5pt

\ \ (S1) $D_{\sigma}$ is a torsion class (i.e. closed for epimorphic images, extensions and coproducts);

\vskip 5pt

\ \ (S2) $X$ lies in $D_{\sigma}$.

\vskip 5pt

$\bullet$ \ silting if there is a projective presentation $\sigma$ of $X$ such that ${\rm Gen}X=D_{\sigma}$, where ${\rm Gen}X$ is the subcategory of  all epimorphic images of modules in ${\rm Add}X$.
\end{defn}

\vskip 10pt

\subsection{\bf Gorenstein silting modules}\ Let $A$ be a Noetherian ring, and $A\mbox{-}{\rm Mod}$ the category of all left $R\mbox{-}$modules. We denote by $A\mbox{-}{\rm GP}$ the full subcategory of $A\mbox{-}{\rm Mod}$ consisting of Gorenstein-projective modules. Assume that $A\mbox{-}{\rm GP}$ is contravariantly finite in $A\mbox{-}{\rm Mod}$. For $A\mbox{-}$modules $M$ and $N$, we compute right derived functors of ${\rm Hom}_{A}(M, N)$ using a proper Gorenstein-projective resolution of $M$ (\cite{EJ2}, \cite{Hol}). We will denote these derived functors by ${\rm Gext}_{A}^{i}(M, N)$. A short exact sequence $0\lxr M\lxr N\lxr L\lxr 0$ is  $G\mbox{-}$exact if and only if it is in ${\rm Gext}_{A}^{1}(L, M)$.

\vskip 5pt

Let $T$ be an $A$-module. Put
$${\rm Gen}_{G}(T):=\{M\in A\mbox{-}{\rm Mod}|\exists \ a \ G\mbox{-}exact \ sequence \ T_{0}\lxr M\lxr 0 \ with \ T_{0}\in {\rm Add}T\}.$$

\vskip 5pt

For a morphism $\theta: G_{1}\lxr G_{0}$ with Gorenstein-projective modules $G_{1}$ and $G_{0}$. We consider the class of $A\mbox{-}$modules
$$D_{\theta}:=\{X\in A\mbox{-}{\rm Mod}\mid {\rm Hom}_{A}(\theta, X) \ {\rm is \ epic}\}.$$

\vskip 5pt

\begin{defn}{\rm (\cite[Definition 2.2]{GMZ})}\ We say that an $A\mbox{-}$module $T$ is

\vskip 5pt

$\bullet$ \ {\bf partial Gorenstein silting} if there is a proper Gorenstein-projective presentation $\theta$ of $T$ such that

\vskip 5pt

\ \ (Gs1) $D_{\theta}$ is a relative torsion class (i.e. closed for $G\mbox{-}$epimorphic images, $G\mbox{-}$extensions and coproducts);

\vskip 5pt

\ \ (Gs2) $T$ lies in $D_{\theta}$.

\vskip 5pt

$\bullet$ \ {\bf Gorenstein silting} if there is a proper Gorenstein-projective presentation $\theta$ of $T$ such that ${\rm Gen}_{G}(T)=D_{\theta}$.
\end{defn}

\vskip 5pt

From \cite[Lemma 2.1]{GMZ}, $D_{\theta}$ is always closed under $G\mbox{-}$epimorphic images, $G\mbox{-}$extensions. Hence $D_{\theta}$ is a relative torsion class if and only if it is closed under coproducts.

\vskip 20pt

\section{\bf Transfer property of silting modules}

\vskip 10pt

In this section, we focus on the transferring properties of silting modules, where the main tools are recollements of module categories and tensor products.

\vskip 10pt

\subsection{\bf Silting modules and idempotent ideals}\ In this subsection, we show that an $A/J\mbox{-}$module $T$ is a silting $A$-module  if and only if $T$  is a silting $A/J$-module whenever $J$ is an idempotent ideal of a ring $A$.

\vskip 10pt

\begin{lem}
\label{silreco}
Let $A,B$ and $C$ be three rings such that there exists the  recollement of module categories
\[
\xymatrix@C=0.5cm{
\ A\mbox{-}{\rm Mod} \ar[rrr]^{\mathsf{i}} &&& B\mbox{-}{\rm Mod} \ar[rrr]^{\mathsf{e}}  \ar @/_1.5pc/[lll]_{\mathsf{q}}  \ar
 @/^1.5pc/[lll]^{\mathsf{p}} &&& C\mbox{-}{\rm Mod}
\ar @/_1.5pc/[lll]_{\mathsf{l}} \ar
 @/^1.5pc/[lll]^{\mathsf{r}}
 }
\]
Then the following hold.

\vskip 5pt

\begin{enumerate}
\item \ $T$ is a silting  $A\mbox{-}$module  if and only if $\mathsf{i}(T)$ is a silting  $B\mbox{-}$module.

\vskip 5pt

\item If $T$ is a silting $B\mbox{-}$module, then $\mathsf{q}(T)$ is a silting $A\mbox{-}$module.

\end{enumerate}
\begin{proof}\ (i). Suppose that $P_{1}\stackrel{\sigma}{\longrightarrow} P_{0}\longrightarrow \mathsf{i}(T)\longrightarrow 0$ is a projective  presentation of $\mathsf{i}(T)$ for an $A\mbox{-}$module $T$. We claim that ${\rm Gen}(T)=D_{\mathsf{q}(\sigma)}$ if and only if ${\rm Gen}(\mathsf{i}(T))=D_{\sigma}$. Indeed, on one hand, since there is the commutative diagram with the isomorphic vertical maps
\[
\xymatrix@C=45pt{
{\rm Hom}_{A}(\mathsf{q}(P_{0}),X) \ar[r]^{}\ar[d]^{\cong} & {\rm Hom}_{A}(\mathsf{q}(P_{1}),X)\ar[d]^{\cong} \\
{\rm Hom}_{B}(P_{0},\mathsf{i}(X))\ar[r]^{} & {\rm Hom}_{A}(P_{1},\mathsf{i}(X)) }
\]
it follows that $X\in D_{\mathsf{q}(\sigma)}$ if and only if $\mathsf{i}(X)\in D_{\sigma}$. On the other hand, $X\in {\rm Gen}(T)$ if and only if there exists an epimorphism $T^{\nu}\lxr X$ for some index set $\nu$
if and only if $\mathsf{i}(T)^{\nu}\lxr \mathsf{i}(X)$ is an epimorphism if and only if $\mathsf{i}(X)\in {\rm Gen}(\mathsf{i}(T))$.

\vskip 5pt

(ii). \ Taking a projective presentation $Q_{1}\stackrel{\kappa}{\longrightarrow} Q_{0}\longrightarrow T\longrightarrow 0$  of $T$, we get that $\mathsf{q}(Q_{1})\stackrel{\mathsf{q}(\kappa)}{\longrightarrow}\mathsf{q}(Q_{0})\longrightarrow \mathsf{q}(T) \longrightarrow 0$ is a projective presentation of $\mathsf{q}(T)$. We prove that $D_{\mathsf{q}(\kappa)}={\rm Gen}(\mathsf{q}(T))$.

\vskip 5pt

Since there are isomorphisms
$${\rm Hom}_{A}(\mathsf{q}(Q_{0}),\ \mathsf{q}(T))\cong {\rm Hom}_{B}(Q_{0},\ \mathsf{i}\mathsf{q}(T))$$
and
$${\rm Hom}_{A}(\mathsf{q}(Q_{1}),\ \mathsf{q}(T))\cong {\rm Hom}_{B}(Q_{1},\ \mathsf{i}\mathsf{q}(T))$$
and the exact sequence $T\longrightarrow \mathsf{i}\mathsf{q}(T) \longrightarrow 0$, we have the following commutative diagram
\[
\xymatrix@C=45pt{
{\rm Hom}_{B}(Q_{0},\ T) \ar[r]^{}\ar[d]^{} & {\rm Hom}_{B}(Q_{0},\ \mathsf{i}\mathsf{q}(T))\ar[r]^{}\ar[d]^{} & 0  \\
{\rm Hom}_{B}(Q_{1},\ T)\ar[r]^{} & {\rm Hom}_{B}(Q_{1},\ \mathsf{i}\mathsf{q}(T))\ar[r]^{} & 0, }
\]
which shows that $\mathsf{q}(T)\in D_{\mathsf{q}(\kappa)}$ and thus ${\rm Gen}(\mathsf{q}(T))\subseteq D_{\mathsf{q}(\kappa)}$.

\vskip 5pt

On the other hand, let $X\in D_{\mathsf{q}(\kappa)}$, there is an epimorphism ${\rm Hom}_{A}(\mathsf{q}(Q_{0}), X)\longrightarrow {\rm Hom}_{A}(\mathsf{q}(Q_{1}),X) \longrightarrow 0$, that is,
\[
\xymatrix@C=45pt{
{\rm Hom}_{A}(\mathsf{q}(Q_{0}),X) \ar[r]^{}\ar[d]^{\cong} & {\rm Hom}_{A}(\mathsf{q}(Q_{1}),X)\ar[r]^{}\ar[d]^{\cong} & 0  \\
{\rm Hom}_{B}(Q_{0},\mathsf{i}(X))\ar[r]^{} & {\rm Hom}_{B}(Q_{1},\mathsf{i}(X))\ar[r]^{} & 0. }
\]
Thus we have that $\mathsf{i}(X)\in D_{\kappa}={\rm Gen}(T)$ and so $X \in \rm{Gen}(\mathsf{q}(T))$.
\end{proof}
\end{lem}

\vskip 10pt

\begin{thm}\
\label{cosiltingstratifying}
Let $A$ be a  ring and $J$ the idempotent ideal of $A$. Let $T$ be a left $A/J\mbox{-}$module.
Then $T$ is a silting $A$-module  if and only if $T$  is a silting $A/J$-module.
\begin{proof}\ The result follows immediately from Lemma~\ref{silreco}.
\end{proof}
\end{thm}

\vskip 10pt

\subsection{\bf Tensor product of silting modules}\ Throughout this subsection, algebras are finite dimensional, and modules are finitely generated.
We show that the tensor products of silting modules are also silting modules.

\vskip 10pt

\begin{thm}
\label{tpofsilting}
Let $A$ and $B$ be two algebras, $\sigma: P_{1}\lxr P_{0}$ the morphism of projective $A\mbox{-}$modules and $\eta: Q_{1}\lxr Q_{0}$ the morphism of projective $B\mbox{-}$modules. Suppose that $T$ is a silting $A\mbox{-}$module with respect to $\sigma$ and $S$ is a silting $B\mbox{-}$module with respect to $\eta$. Then $T\otimes_{k} S$ is a silting $A\otimes_{k} B\mbox{-}$module with respect to $\sigma\otimes \eta$.
\end{thm}
\begin{proof}\ From the exact sequences
$$P_{1}\stackrel{\sigma}{\lxr} P_{0}\lxr T\lxr 0 \ \ \ and\ \ \ Q_{1}\stackrel{\eta}{\lxr} Q_{0}\lxr S\lxr 0,$$
we get the induced exact sequence
$$P_{1}\otimes_{k} Q_{1}\xrightarrow{\sigma\otimes \eta} P_{0}\otimes_{k} Q_{0}\lxr T\otimes_{k}S\lxr 0.$$
Since there are the following isomorphisms
$${\rm Hom}_{A\otimes_{k}B}(P_{i}\otimes_{k} Q_{i}, T\otimes_{k}S)\cong {\rm Hom}_{A}(P_{i}, T)\otimes_{k}{\rm Hom}_{B}(Q_{i}, S),$$
for $i=0, 1$, it follows that $T\otimes_{k}S\in D_{\sigma\otimes \eta}$ if and only if $T\in D_{\sigma}$ and $S\in D_{\eta}$. This is to say, $D_{\sigma\otimes \eta}=D_{\sigma}\otimes D_{\eta}$.

\vskip 5pt

Since $D_{\sigma}$ and $D_{\eta}$ are closed under direct sums, it is obvious that $D_{\sigma\otimes \eta}$ has the same property, i.e., $D_{\sigma\otimes \eta}$ is a torsion class. Furthermore, for every $M\otimes_{k}N\in D_{\sigma\otimes \eta}$, we have that
$$M\in D_{\sigma}={\rm Gen}T \ \ \ \ and \ \ \ \ N\in D_{\eta}={\rm Gen}S.$$
Therefore, $M\otimes_{k}N\in {\rm Gen}T\otimes_{k}{\rm Gen}S\subseteq{\rm Gen}(T\otimes_{k}S)$. This completes the proof.
\end{proof}

\vskip 20pt

\section{\bf Transfer property of Gorenstein silting modules}

\vskip 10pt

In this section, we focus on the transferring properties of Gorenstein silting modules, where the main tool is the recollement of module categories.
 As a consequence, we characterise the relations of Gorenstein silting modules among the rings $A, B$ and $\Gamma$, where $\Gamma$ is the upper triangular matrix rings building from $A$ and $B$.

\vskip 10pt

Now let $A,\ B$ and $C$ be three Noetherian rings such that $A\mbox{-}{\rm GP}$ and $C\mbox{-}{\rm GP}$ are contravariantly finite in $A\mbox{-}{\rm Mod}$ and $C\mbox{-}{\rm Mod}$, respectively. Suppose that there exists the recollement of module categories
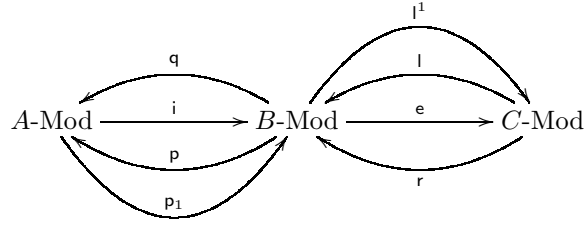
\[
\xymatrix@C=0.5cm{
A\mbox{-}{\rm Mod}\ar[rrr]^{\mi}  \ar @/_3.0pc/[rrr]^{\map_{1}}  &&& B\mbox{-}{\rm Mod}\ar @/^3.0pc/[rrr]^{\ml^{1}} \ar[rrr]^{\me}  \ar @/_1.5pc/[lll]_{\mq}  \ar @/^1.5pc/[lll]_{\map}   &&& C\mbox{-}{\rm Mod}
\ar @/_1.5pc/[lll]_{\ml} \ar @/^1.5pc/[lll]^{\mr}
 }
\]
such that $\map$ has an exact right adjoint $\map_{1}$ and $\ml$ has a right adjoint $\ml^{1}$ such that $\ml^{1}\mi$ preserves Gorenstein-projective modules. Let $X\in A\mbox{-}{\rm Mod}$ and $Y\in C\mbox{-}{\rm Mod}$, and
$$E_{1}\stackrel{\theta_{X}}{\lxr} E_{0}\lxr X\lxr 0\eqno(1)$$
 and
$$G_{1}\stackrel{\theta_{Y}}{\lxr} G_{0}\lxr Y\lxr 0\eqno(2)$$
the proper Gorenstein-projective presentations of $X$ and $Y$ respectively. Applying the functors $\mi(-)$ and $\ml(-)$ to (1) and (2) respectively, we get the following exact sequences:
$$\mi(E_{1})\xrightarrow{\mi(\theta_{X})} \mi(E_{0})\lxr \mi(X)\lxr 0$$
and
$$\ml(G_{1})\xrightarrow{\ml(\theta_{Y})} \ml(G_{0})\lxr \ml(Y)\lxr 0.$$
Then we get the following exact sequence:
$$\mi(E_{1})\oplus \ml(G_{1})\stackrel{\theta}{\lxr} \mi(E_{0})\oplus \ml(G_{0})\lxr \mi(X)\oplus \ml(Y)\lxr 0\eqno(3)$$
with $\theta=\begin{pmatrix}\begin{smallmatrix}
\mi(\theta_{X}) & 0 \\
0 & \ml(\theta_{Y}) \\
\end{smallmatrix}\end{pmatrix}$.

\vskip 10pt

\begin{lem}
\label{GP}
 The exact sequence {\rm (3)} mentioned above is a proper Gorenstein-projective presentation of $\mi(X)\oplus \ml(Y)$.
\end{lem}
\begin{proof}\ By assuming that $\map$ is an exact functor, it follows that $\mi$ preserves projective modules and $\mr$ admits a right adjoint $\mr_{1}$. Let $T$ be a projective $B\mbox{-}$module, then there is an $A\mbox{-}$module $Z$ such that there exists the following exact sequence:
\[
\xymatrix@C=25pt{
0\ar[r]  &   \mi(Z)\ar[r]   &   \ml\me(T)\ar[rr]\ar@{->>}[rd]  & &  T\ar[r]  &  \mi\mq(T)\ar[r]   &   0.  \\
& & &  K\ar@{^{(}->}[ru]  & & &}
\]
By $\mi\mq(T)$ being projective, we have that $T\cong K\oplus \mi\mq(T)$ and $\ml\me(T)\cong K\oplus \mi(Z)$ with $\me(K)\cong \me(T)$. This means that the projective $B\mbox{-}$modules are of the form either $\mi(P)$ or $\ml(Q)$, where $P$ runs over projective $A\mbox{-}$modules and $Q$ runs over projective $C\mbox{-}$modules.

\vskip 5pt

With $P^{\prime}$ being $A$-projective and $Q^{\prime}$ being $C$-projective, since there are the isomorphisms
$${\rm Hom}_{B}(\mi(P^{\prime}),\ \mi(P)\oplus \ml(Q))\cong {\rm Hom}_{A}(P^{\prime}, P)\oplus {\rm Hom}_{A}(P^{\prime}, \map\ml(Q))$$
and
$${\rm Hom}_{B}(\ml(Q^{\prime}),\ \mi(P)\oplus \ml(Q))\cong {\rm Hom}_{C}(Q^{\prime}, Q),$$
we get that each indecomposable Gorenstein-projective $B\mbox{-}$module  is of the form $\mi(E)$ or $\ml(G)$ whenever $E$ runs over all the indecomposable Gorenstein-projective $A\mbox{-}$modules and $G$ runs over all the indecomposable Gorenstein-projective $C\mbox{-}$modules. Therefore, $\mi(E_{1})\oplus \ml(G_{1})$ and $\mi(E_{0})\oplus \ml(G_{0})$, mentioned in the exact sequence (3), are Gorenstein-projective $B\mbox{-}$modules.

\vskip 5pt

Now we prove the exact sequence (3) is ${\rm Hom}_{B}(S, -)\mbox{-}$exact for any Gorenstein-projective $B\mbox{-}$module $S=\mi(E)\oplus \ml(G)$. Since there are the following commutative diagrams, where we denote ${\rm Hom}(-,-)$ by $(-,-)$ simply:
\[\xymatrix@C=10pt{
(\mi(E),\ \mi(E_{1})\oplus \ml(G_{1})) \ar[d]_{\cong}\ar[r]^{} &   (\mi(E),\ \mi(E_{0})\oplus \ml(G_{0})) \ar[d]_{\cong}\ar[r]^{}  &  (\mi(E),\ \mi(X)\oplus \ml(Y)) \ar[d]_{\cong}\\
(E, E_{1})\oplus (\ml^{1} \mi(E), G_{1}) \ar[r]^{} &   (E, E_{0})\oplus (\ml^{1} \mi(E), G_{0})\ar[r]^{}   &  (E, X)\oplus (\ml^{1} \mi(E), Y) }
\]
\noindent and
\[\xymatrix@C=10pt{
(\ml(G),\ \mi(E_{1})\oplus \ml(G_{1})) \ar[d]_{\cong}\ar[r]^{} &   (\ml(G),\ \mi(E_{0})\oplus \ml(G_{0})) \ar[d]_{\cong}\ar[r]^{}  &   (\ml(G),\ \mi(X)\oplus \ml(Y)) \ar[d]_{\cong}\\
(G, G_{1}) \ar[r]^{} &   (G, G_{0})\ar[r]^{}   &  (G, Y)}
\]
it is obvious that the sequence (3) is  a $G\mbox{-}$exact sequence.
\end{proof}

\vskip 10pt

\begin{lem}\
\label{torsion class}
Let $X\in A\mbox{-}{\rm Mod},\ Y\in C\mbox{-}{\rm Mod}$, and $Z\in B\mbox{-}{\rm Mod}$. The following statements hold.
\vskip 5pt
\begin{enumerate}
\item \ $Z\in D_{\theta}$ if and only if $\map(Z)\in D_{\theta_{X}}$ and $\me(Z)\in D_{\theta_{Y}}$.

\vskip 5pt

\item \ If $X\in D_{\theta_{X}}$, then $\mi(X)\in D_{\theta}$.

\vskip 5pt

\item \ If $Y\in D_{\theta_{Y}}$ and $\map\ml(Y)\in D_{\theta_{X}}$, then $\ml(Y)\in D_{\theta}$.

\vskip 5pt

\item \ $\mi(X)\oplus \ml(Y)\in D_{\theta}$ if and only if $X\in D_{\theta_{X}},\ Y\in D_{\theta_{Y}}$ and $\map\ml(Y)\in D_{\theta_{X}}$.

\vskip 5pt

\item \ $D_{\theta}$ is a relative torsion class if and only if both $D_{\theta_{X}}$ and $D_{\theta_{Y}}$ are relative torsion classes.
\end{enumerate}
\end{lem}
\begin{proof}\ Let $Z\in D_{\theta}$. By the definition of $D_{\theta}$, there is the following commutative diagram:
\[\xymatrix@C=15pt{
{\rm Hom}_{B}(\mi(E_{0})\oplus \ml(G_{0}),\ Z) \ar[d]_{\cong}\ar@{->>}[r]^{} &   {\rm Hom}_{B}(\mi(E_{1})\oplus \ml(G_{1}),\ Z) \ar[d]_{\cong}\\
{\rm Hom}_{A}(E_{0}, \map(Z))\oplus {\rm Hom}_{C}(G_{0}, \me(Z)) \ar@{->>}[r]^{} &   {\rm Hom}_{A}(E_{1}, \map(Z))\oplus {\rm Hom}_{C}(G_{1}, \me(Z)),  }
\]
this induces the exact sequences:
$${\rm Hom}_{A}(E_{0},\ \map(Z)))\lxr {\rm Hom}_{A}(E_{1},\ \map(Z))\lxr 0$$
and
$${\rm Hom}_{C}(G_{0},\ \me(Z))\lxr {\rm Hom}_{C}(G_{1},\ \me(Z))\lxr 0,$$
that is, $\map(Z)\in D_{\theta_{X}}$ and $\me(Z)\in D_{\theta_{Y}}$. Conversely, if $\map(Z)\in D_{\theta_{X}}$ and $\me(Z)\in D_{\theta_{Y}}$, we can get from using the above commutative diagram once again that $Z\in D_{\theta}$.

\vskip 5pt

(ii)-(v) follow directly from (i).
\end{proof}

\vskip 10pt

\begin{prop}
\label{partialGS}
$\mi(X)\oplus \ml(Y)$ is a  partial Gorenstein silting $B\mbox{-}$module if and only if $X$ and $\map\ml(Y)$ are partial  Gorenstein silting $A\mbox{-}$modules and $Y$ is a partial  Gorenstein silting $B\mbox{-}$module.
\end{prop}
\begin{proof}\ The result is immediately from Lemma~\ref{GP} and Lemma~\ref{torsion class}.
\end{proof}

\vskip 10pt

Now we apply it to the triangular matrix ring. In the following, let $\Gamma=\begin{pmatrix}\begin{smallmatrix}
A & N \\
0 & B \\
\end{smallmatrix}\end{pmatrix}$ be a Gorenstein ring such that ${\rm gl.dim}A<\infty$ and $_{A}N$ and $N_{B}$ are projective. Then there is the following recollement:
\[
\xymatrix@C=0.5cm{
A\mbox{-}{\rm Mod}\ar[rrr]^{Z_{A}}  \ar @/_3.0pc/[rrr]^{H_{A}}  &&& \Gamma\mbox{-}{\rm Mod}\ar @/^3.0pc/[rrr]^{} \ar[rrr]^{U_{B}}  \ar @/_1.5pc/[lll]_{}  \ar @/^1.5pc/[lll]_{U_{A}}   &&& B\mbox{-}{\rm Mod}
\ar @/_1.5pc/[lll]_{T_{B}} \ar @/^1.5pc/[lll]^{}
 }
\]
where $Z_{A}(X)=(X, 0, 0), \ U_{A}(X, Y, f)=X, \ T_{B}(Y)=(N\otimes_{B}Y, Y, 1)$ and $H_{A}(X)=(X, {\rm Hom}_{A}(N, X), \varepsilon_{X})$ with $\varepsilon$ being the counit of $(N\otimes_{B}-, {\rm Hom}_{A}(N, -))$.

\vskip 5pt

\vskip 10pt

Let $X\in A\mbox{-}{\rm Mod}$ and $Y\in B\mbox{-}{\rm Mod}$, and let
$$P_{1}\xrightarrow{\theta_{X}} P_{0}\longrightarrow X\longrightarrow 0 $$
and
$$E_{1}\xrightarrow{\theta_{Y}} E_{0}\longrightarrow Y\longrightarrow 0 $$
be the projective presentation of $X$ and the proper Gorenstein-projective presentation of $Y$ respectively.
Then by Lemma~\ref{GP} the  sequence
$$T^{\bullet}:(P_{1}, 0)\oplus (N\otimes_{B}E_{1}, E_{1}) \stackrel{\theta}{\rightarrow} (P_{0}, 0)\oplus (N\otimes_{B}E_{0}, E_{0})\rightarrow (X, 0)\oplus (N\otimes_{B}Y, Y) \rightarrow 0 $$
is the proper Gorenstein-projective presentation of $(X, 0)\oplus (N\otimes_{B}Y, Y)$.

\vskip 10pt

As an immediate consequence of Proposition~\ref{partialGS}, we have the following.

\vskip 5pt

\begin{cor}\
\label{partial}
$(X, 0)\oplus (N\otimes_{B}Y, Y)$ is a partial Gorenstein silting module with respect to $\theta$ if and only if $X$ is a partial Gorenstein silting $A\mbox{-}$module with respect to $\theta_{X}$, and $Y$ is a partial Gorenstein silting $B\mbox{-}$module with respect to $\theta_{Y}$ such that $N\otimes_{B}Y\in {\rm Gen}_{G}X$.
\end{cor}

\vskip 10pt

It can be seen from the following theorem that the Gorenstein silting modules cannot be directly converted over $\Gamma$ and $A,\ B$, and additional conditions are required.

\vskip 10pt

\begin{lem}
\label{partialtosilting}
{\rm (\cite[Proposition 2.3]{GMZ})}\ Let $T$ be an $A$-module with proper Gorenstein-projective presentation $\theta: G_{1}\lxr G_{0}$. If $T$ is a partial Gorenstein silting module with respect to $\theta$, and for each $P\in A\mbox{-}{\rm GP}$, there exists a $G\mbox{-}$exact sequence $P\stackrel{\phi}{\lxr} T_{0}\lxr T_{-1}\lxr 0$ with $T_{0}$ and $T_{-1}$ in ${\rm Add}T$ such that $\phi$ is the left $D_{\theta}\mbox{-}$approximation, then $T$ is a Gorenstein silting module.
\end{lem}

\vskip 10pt

\begin{thm}\
\label{Gsilting}
Let $X\in A\mbox{-}{\rm Mod}$ and $Y\in B\mbox{-}{\rm Mod}$. Then the following are equivalent:
\vskip 5pt
\begin{enumerate}
\item {\rm (a)}\ $(X, 0)\oplus (N\otimes_{B}Y, Y)$ is a Gorenstein silting $\Gamma\mbox{-}$module;

\vskip 5pt

\noindent{\rm (b)}\ there exists a $G\mbox{-}$exact sequence $(P, 0)\oplus (N\otimes_{B}E, E)\stackrel{\lambda}{\lxr} (X_{0}, 0)\oplus (N\otimes_{B}Y_{0}, Y_{0})\lxr (X_{-1}, 0)\oplus (N\otimes_{B}Y_{-1}, Y_{-1})\lxr 0$ with $\lambda=\begin{pmatrix}\begin{smallmatrix}
(\phi, 0) & 0 \\
0 & ({\rm Id}_{N}\otimes \psi, \psi) \\
\end{smallmatrix}\end{pmatrix}$, $X_{i}\in {\rm Add}X$ and $Y_{i}\in {\rm Add}Y$ such that $\lambda$ is the left $D_{\theta}\mbox{-}$approximation, $i=-1,0$, for the Gorenstein-projective $\Gamma\mbox{-}$module $(P, 0)\oplus (N\otimes_{B}E, E)$.

\vskip 5pt

\item {\rm (c)}\ $X$ is a Gorenstein silting $A\mbox{-}$module and $N\otimes_{B}Y\in {\rm Gen}_{G}X$;

\vskip 5pt

\noindent{\rm (d)}\ $Y$ is a Gorenstein silting $B\mbox{-}$module;

\vskip 5pt

\noindent{\rm (e)}\ there exists an exact sequence $P\stackrel{\phi}{\lxr} X_{0}\lxr X_{-1}\lxr 0$ with $X_{0}$ and $X_{-1}$ in ${\rm Add}X$ such that $\phi$ is the left $D_{\theta_{X}}\mbox{-}$approximation, for each $P\in A\mbox{-}{\rm P}$;

\vskip 5pt

\noindent{\rm (f)}\ there exists a $G\mbox{-}$exact sequence $E\stackrel{\psi}{\lxr} Y_{0}\lxr Y_{-1}\lxr 0$ with $Y_{0}$ and $Y_{-1}$ in ${\rm Add}Y$ such that $\psi$ is the left $D_{\theta_{Y}}\mbox{-}$approximation, for each $E\in B\mbox{-}{\rm GP}$.
\end{enumerate}
\end{thm}
\begin{proof}\ (i)$\Longrightarrow$(ii)\ Let $(X, 0)\oplus (N\otimes_{B}Y, Y)$ be a Gorenstein silting $\Gamma\mbox{-}$module. Then
$$D_{\theta}={\rm Gen}_{G}((X, 0)\oplus (N\otimes_{B}Y, Y)).$$
Moreover, by Corollary~\ref{partial}, ${\rm Gen}_{G}X\subseteq D_{\theta_{X}}$, ${\rm Gen}_{G}Y\subseteq D_{\theta_{Y}}$, $N\otimes_{B}Y\in D_{\theta_{X}}$. And $(K, 0)\in D_{\theta}$ if $K\in D_{\theta_{X}}$ by Lemma~\ref{torsion class}. Now let $K\in D_{\theta_{X}}$. Then there is  a $G\mbox{-}$epimorphism
$$((X, 0)\oplus (N\otimes_{B}Y, Y))^{I}\lxr (K, 0)$$
for an index set $I$. If $(K, 0)$ has a direct summand $(K', 0)$ such that there is a $G\mbox{-}$epimorphism $(N\otimes_{B}Y, Y)^{J}\lxr (K', 0)$ for some index set $J$. Then there is the following commutative diagram with $G\mbox{-}$epic columns
\[\xymatrix@C=50pt{
(N\otimes_{B}Y)^{I} \ar[d]_{0} \ar[r]^{{\rm Id}} &   (N\otimes_{B}Y)^{I} \ar[d]^{\kappa}     \\
0  \ar[r]^{} & K',  }
\]
and so $K'=0$. This implies that $(K, 0)\in {\rm Gen}_{G}(X, 0)$ and  $D_{\theta_{X}}={\rm Gen}_{G}X$. Thus $X$ is a Gorenstein silting $A\mbox{-}$module and $N\otimes_{B}Y\in D_{\theta_{X}}={\rm Gen}_{G}X$. On the other hand, let $L\in D_{\theta_{Y}}$, then $(0, L)\in D_{\theta}$. Hence there is a $G\mbox{-}$epimorphism
$$((X, 0)\oplus (N\otimes_{B}Y, Y))^{I}\lxr (0, L)$$
for an index set $I$. Then we can get that there is a $G\mbox{-}$epimorphism $(N\otimes_{B}Y, Y)^{I}\lxr (0, L)$, and so we have a $G\mbox{-}$epimorphism $Y^{I}\lxr  L$. This implies that $D_{\theta_{Y}}\subseteq {\rm Gen}_{G}Y$. Thus $Y$ is a Gorenstein silting $B\mbox{-}$module.

\vskip 5pt

From (b), we can get the exact sequence $P\stackrel{\phi}{\lxr} X_{0}\lxr X_{-1}\lxr 0$ with $X_{i}\in {\rm Add}X$ and the $G\mbox{-}$exact sequence $E\stackrel{\psi}{\lxr} Y_{0}\lxr Y_{-1}\lxr 0$ with $Y_{i}\in {\rm Add}Y$ for $i=-1,0$, where $P\in A\mbox{-}{\rm P}$ and $E\in B\mbox{-}{\rm GP}$. We claim that $\phi$ is the left $D_{\theta_{X}}\mbox{-}$approximation and $\psi$ is the left $D_{\theta_{Y}}\mbox{-}$approximation. In fact, let $U\in D_{\theta_{X}}$ and $a_{1}\in {\rm Hom}_{A}(P, U)$. Then $(U, 0)\in D_{\theta}$ and $((a_{1}, 0),\ 0)\in {\rm Hom}_{\Gamma}((P, 0)\oplus (N\otimes_{B}E, E),\ (U, 0))$. Thus there exists $((a'_{1}, 0),\ v)\in {\rm Hom}_{\Gamma}((X_{0}, 0)\oplus (N\otimes_{B}Y_{0}, Y_{0}),\ (U, 0))$ such that the following diagram commutes:
\[
\xymatrix@C=80pt{
(P, 0)\oplus (N\otimes_{B}E, E) \ar[d]_{((a_{1}, 0),\ 0)} \ar[r]^{\lambda} &   (X_{0}, 0)\oplus (N\otimes_{B}Y_{0}, Y_{0})\ar@{-->}[ld]^{((a'_{1}, 0),\ v)}   \\
(U, 0).   &      }
\]
Then we get from $((a'_{1}, 0),\ v)\circ \lambda=((a_{1}, 0),\ 0)$ that $a'_{1}\circ \phi=a_{1}$. This implies that $\phi$ is a left $D_{\theta_{X}}\mbox{-}$approximation. Similarly, we can prove that $\psi$ is the left $D_{\theta_{Y}}\mbox{-}$approximation.

\vskip 5pt

(ii)$\Longrightarrow$(i)\ We know from Corollary~\ref{partial} that $(X, 0)\oplus (N\otimes_{B}Y, Y)$ is a partial Gorenstein silting module. Since there are the $G\mbox{-}$exact sequences
$$G^{\bullet}: P\stackrel{\phi}{\lxr} X_{0}\lxr X_{-1}\lxr 0$$
and
$$E^{\bullet}: E\stackrel{\psi}{\lxr} Y_{0}\lxr Y_{-1}\lxr 0,$$
we get an exact sequence
$$S^{\bullet}: (P, 0)\oplus (N\otimes_{B}E, E)\stackrel{\lambda}{\rightarrow} (X_{0}, 0)\oplus (N\otimes_{B}Y_{0}, Y_{0})\rightarrow (X_{-1}, 0)\oplus (N\otimes_{B}Y_{-1}, Y_{-1})\rightarrow 0$$
with $\lambda=\begin{pmatrix}\begin{smallmatrix}
(\phi, 0) & 0 \\
0 & ({\rm Id}_{N}\otimes \psi, \psi) \\
\end{smallmatrix}\end{pmatrix}$. Moreover, since each Gorenstein-projective $\Gamma\mbox{-}$module is of the form $(P, 0)\oplus (N\otimes_{B}E, E)$, we get that ${\rm Hom}_{\Gamma}((P, 0)\oplus (N\otimes_{B}E, E),\ S^{\bullet})$ is exact. Therefore, $S^{\bullet}$ is $G\mbox{-}$exact.

\vskip 5pt

We next prove that $\lambda$ is the left $D_{\theta}\mbox{-}$approximation. Let $(U, V, h)\in D_{\theta}$ and $((a_{1}, 0),\ (a_{2}, b_{2}))\in {\rm Hom}_{\Gamma}((P, 0)\oplus (N\otimes_{B}Q, Q),\ (U, V, h))$. Then there exist the following commutative diagrams below:
\[
\xymatrix@C=80pt{
(P, 0)\oplus (N\otimes_{B}Q, Q) \ar[d]_{((a_{1}, 0),\ (a_{2}, b_{2}))} \ar[r]^{\lambda} &   (X_{0}, 0)\oplus (N\otimes_{B}Y_{0}, Y_{0})\ar@{-->}[ld]^{}   \\
(U, V, h).   &      }
\]
and
\[
\xymatrix@C=60pt{
N\otimes_{B}Q \ar[d]_{{\rm Id}_{N}\otimes b_{2}} \ar[r]^{{\rm Id}_{N\otimes Q}} &   N\otimes_{B}Q \ar[d]^{a_{2}}     \\
N\otimes_{B}V  \ar[r]^{h} & U.   }
\]
Moreover, $U\in D_{\theta_{X}}$ and $V\in D_{\theta_{Y}}$ by Lemma~\ref{torsion class}. Since $\phi$ is a left $D_{\theta_{X}}\mbox{-}$approximation and $\psi$ is a left $D_{\theta_{Y}}\mbox{-}$approximation, there exist $a'_{1}\in {\rm Hom}_{A}(X_{0}, U)$ and $b'_{2}\in {\rm Hom}_{B}(Y_{0}, V)$ respectively such that $a_{1}= a'_{1}\phi$ and $b_{2}=b'_{2}\psi$, as shown in the diagram below:
\[
\xymatrix@C=50pt{
P \ar[d]_{a_{1}} \ar[r]^{\phi} &   X_{0} \ar@{-->}[ld]^{a'_{1}}   \\
U  &      }\ \ \ \ \ \ \ \ \ \ \ \ \ \ \ \ \ \
\xymatrix@C=50pt{
Q \ar[d]_{b_{2}} \ar[r]^{\psi} &   Y_{0} \ar@{-->}[ld]^{b'_{2}}   \\
V  &      }
\]
Then we can get the following equalities:
$$\begin{aligned}
((a'_{1}, 0),\ (h({\rm Id}_{N}\otimes b_{2}), b'_{2}))\circ \lambda &= ((a'_{1}, 0),\ (h({\rm Id}_{N}\otimes b_{2}), b'_{2}))\circ \begin{pmatrix}\begin{smallmatrix}
(\phi, 0) & 0 \\
0 & ({\rm Id}_{N}\otimes \psi, \psi) \\
\end{smallmatrix}\end{pmatrix}\\
&= ((a'_{1}\phi, 0),\ (h({\rm Id}_{N}\otimes b_{2})({\rm Id}_{N}\otimes \psi), b'_{2}\psi))\\
&= ((a_{1}, 0),\ (a_{2}, b_{2})).
\end{aligned}$$
This implies that $\lambda$ is a left $D_{\theta}\mbox{-}$approximation. Therefore we can get from Lemma~\ref{partialtosilting} that $(X, 0)\oplus (N\otimes_{B}Y, Y)$ is a Gorenstein silting $\Gamma\mbox{-}$module.
\end{proof}

\end{document}